\documentclass{amsart}

\usepackage{fancyhdr}
\usepackage{adjustbox}
\usepackage{mathrsfs}
\usepackage{overpic}
\usepackage{tikz}
\usepackage{amsopn}
\usepackage{tikz-cd}
\usepackage{pgf}
\usepackage[all,cmtip]{xy}
\usepackage{multicol}
\usepackage{marginnote}
\usepackage[margin=3.4cm]{geometry}
\usepackage{amscd}
\usepackage{amsmath,amsfonts,amsthm,amssymb}
\usepackage{setspace}
\usepackage{Tabbing}
\usepackage{fancyhdr}
\usepackage{lastpage}
\usepackage{extramarks}
\usepackage{chngpage}
\usepackage{soul,color}
\usepackage{graphicx,float,wrapfig}
\graphicspath{ {./diagrams/} }
\usepackage{url}
\usepackage{chngcntr}
\usepackage{cite}
\usepackage[pagewise]{lineno}
\usepackage{arydshln}
\usepackage{mathtools}
\usepackage{hyperref}
\hypersetup{
    colorlinks=true,
    pagebackref=true,
    linkcolor=black,
    filecolor=black,      
    urlcolor=black,
    citecolor=black
}

\newcommand\coolover[2]{\mathrlap{\smash{\overbrace{\phantom{%
    \begin{matrix} #2 \end{matrix}}}^{\mbox{$#1$}}}}#2}

\newcommand\coolunder[2]{\mathrlap{\smash{\underbrace{\phantom{%
    \begin{matrix} #2 \end{matrix}}}_{\mbox{$#1$}}}}#2}

\newcommand\coolleftbrace[2]{%
#1\left\{\vphantom{\begin{matrix} #2 \end{matrix}}\right.}

\newcommand{\R}{\mathbb{R}}

\newcommand{\J}{\mathbb{J}}

\newcommand{\Q}{\mathbb{Q}}
\newcommand{\Z}{\mathbb{Z}}
\newcommand{\s}{\mathcal{S}}

\newcommand{\Disk}{\mathcal{D}\mathsf{isk}}

\newcommand{\p}{\mathfrak{p}}

\newcommand{\tZ}{{\widetilde{Z}}}
\newcommand{\tZU}{{\widetilde{Z}/U}}
\newcommand{\sd}{{SD}}

\newcommand{\id}{{\rm id}}

\renewcommand{\H}{\mathcal{H}}
\renewcommand{\S}{\mathbb{S}}
\newcommand{\sS}{\mathscr{S}}

\newcommand{\h}{\mathcal{H}}
\newcommand{\wh}{{\rm Wh}}

\newcommand{\diff}{{\rm Diff}}

\newcommand{\bdiff}{{\rm BDiff}}

\newcommand{\btop}{{\rm BTop}}
\newcommand{\bo}{{\rm BO}}

\newcommand{\Th}{{\rm Th}}

\DeclareMathOperator*{\holim}{\sf holim}

\DeclareMathOperator{\disk}{\sf Disk}

\DeclareMathOperator{\tr}{\sf tr}
\DeclareMathOperator{\ind}{\sf Ind}

\theoremstyle{plain}
\newtheorem{theorem}{Theorem}
\newtheorem*{theorem*}{Theorem}
\newtheorem*{maintheorem*}{Main Theorem}
\newtheorem*{dualitytheorem*}{Duality Theorem}
\newtheorem*{vanishingtheorem*}{Vanishing Theorem}

\newcounter{lemma}[section]

\theoremstyle{definition}\newtheorem{lemma}[lemma]{Lemma}

\theoremstyle{plain}\newtheorem{numberedtheorem}[lemma]{Theorem}

\theoremstyle{plain}\newtheorem{corollary}[lemma]{Corollary}
\theoremstyle{plain}
\theoremstyle{definition}\newtheorem{example}[lemma]{Example}
\theoremstyle{definition}\newtheorem{prop}[lemma]{Proposition}
\theoremstyle{definition}\newtheorem{rmk}[lemma]{Remark}
\theoremstyle{definition}\newtheorem{defn}[lemma]{Definition}
\theoremstyle{definition}
\theoremstyle{definition}

\title{Duality and Vanishing Theorems for Topologically Trivial Families of Smooth h-cobordisms}
\author{Yajit Jain}

\thanks{The author is supported by NSF Grant No. DMS-2103276.
}

\email{yajit\_jain@brown.edu}
\address{Department of Mathematics, Brown University}

\date{\today}

\begin{document}
\maketitle
\begin{abstract}
Using the work of Dwyer, Weiss, and Williams we associate an invariant to any topologically trivial family of smooth h-cobordisms. This invariant is called the smooth structure class, and is closely related to the higher Franz–Reidemeister torsion of Igusa and Klein. We compute the smooth structure class in terms of a fiberwise generalized Morse function using fiberwise Poincar\'e–Hopf theory. This computation gives rise to a duality theorem for the smooth structure class that generalizes Milnor's duality theorem for the Whitehead torsion. From this result we deduce a vanishing theorem that implies the Rigidity Conjecture of Goette and Igusa. This conjecture states that, after rationalizing, there are no stable exotic smoothings of manifold bundles with closed even dimensional fibers. 
\end{abstract}

\tableofcontents

\section{Introduction}

Recall the h-cobordism theorem of Smale and the s-cobordism theorem of Barden, Mazur, and Stallings. For a connected smooth manifold $M$ of dimension $d\ge 5$ these assert that the Whitehead torsion gives a bijection 
\begin{align*}
\{\text{h-cobordisms on $M$}\}/\text{diffeo.}&\xrightarrow{\cong} \wh_1(\pi_1(M))\\
[W]&\mapsto\tau(W,M)
\end{align*}
between h-cobordisms on $M$ up to diffeomorphism and the Whitehead group. 

\medskip

The stable parametrized h-cobordism theorem~\cite{WJR13} is a generalization of this classical result into the setting of infinite loop spaces. It provides a homotopy equivalence between the stable parametrized smooth h-cobordism space and the loopspace of the smooth Whitehead space:
$$
\h(M)\simeq \Omega\wh^\diff(M)
$$
Smooth manifold bundles over a base $B$ whose fibers are h-cobordisms are classified by maps from $B$ into $\h(M)$. 
Using the homotopy equivalence above, we can associate to such bundles a map $B\to \Omega\wh^\diff(M)$, which we refer to as a \textit{higher torsion invariant}. 
Using the fact that $\pi_0\wh^\diff(M)\cong\wh_1(\pi_1(M))$, we see that when $B=*$, the higher torsion invariant is equivalent to the Whitehead torsion. 
One can also construct higher torsion invariants for bundles with closed acyclic fibers, or under  conditions on the monodromy action on the fiber homology. 

\smallskip

In this paper we are concerned with \textit{topologically trivial families of smooth h-cobordisms}. Such a family is defined to be a bundle of smooth h-cobordisms that is fiberwise homeomorphic, but not necessarily fiberwise diffeomorphic, to a bundle of product h-cobordisms. 
It follows from work of Burghelea and Lashof~\cite{BL77} (in the unparameterized setting), as well as Dwyer, Weiss, and Williams~\cite{DWW03} (in the parameterized setting) that there is a specialization of the stable parameterized h-cobordism theorem to such h-cobordism bundles. 
In particular, for a smooth manifold bundle $p:M\to B$, there is a homotopy equivalence
\begin{align*}
\h_B^{t/d}(M)&\xrightarrow{\simeq} \Gamma_B\h^\%_B(M)
\end{align*}
between the space of topologically trivial families of smooth h-cobordisms on $M$ and the space of sections of the parameterized ring-spectrum obtained by taking a fiberwise smash product of $M$ with the stable parameterized h-cobordism space of a point (see \cite{GIW14} for details). 
Given a topologically trivial family of smooth h-cobordisms $W$, the homotopy equivalence gives rise to an invariant $\theta(W,M)\in \Gamma_B\h^\%_B(M)$ which we refer to as the \textit{smooth structure characteristic}, and its associated rational component $\Theta(W,M)\in\pi_0\Gamma_B\h_B^\%(M)\otimes\Q$ is the \textit{smooth structure class}. 
The smooth structure characteristic is a variation on the higher torsion invariant from above, and is closely related to the higher Franz–Reidemeister torsion of Igusa and Klein, as well as to the higher smooth torsion of Dwyer, Weiss, and Williams.

\smallskip

We prove two main theorems in this paper. The first is a duality theorem for the smooth structure class which generalizes Milnor's duality theorem for Whitehead torsion to the setting of topologically trivial families of smooth h-cobordisms. The second is a vanishing theorem that resolves the Rigidity Conjecture of Goette and Igusa~\cite{GI14} (see Subsection~\ref{subsec:comparison_to_GIW} for details).

\begin{dualitytheorem*}[Theorem~\ref{thm:duality}]
For $p:W\to B$ a topologically trivial bundle of smooth h-cobordisms with fiber dimension $n$,
$$
\Theta(W,\partial_0W) = (-1)^{n-1}\Theta(W,\partial_1W)
$$
\end{dualitytheorem*}
\begin{vanishingtheorem*}[Theorem~\ref{thm:rigidity}]
If the fibers of $p_0:M\to B$ are even dimensional and closed, then for any topologically trivial family of smooth h-cobordisms $p:W\to B$ from $M$ to $M'$, $\Theta(W,M)-\Theta(W,M')$ is trivial. 
\end{vanishingtheorem*}

\medskip

In the remainder of this introduction, we contextualize the duality and vanishing theorems above as manifestations of similar theorems for related invariants including the Euler characteristic, Reidemeister torsion, the Becker–Gottlieb transfer, and the higher Franz–Reidemeister torsion. Vanishing theorems for each of these invariants are proven in stylistically equivalent ways, though the technical ingredients vary widely. They generally follow by applying a Poincar\'e–Hopf type theorem in combination with Poincar\'e duality. In order to motivate the contents of this paper, we will briefly describe these theorems and their proofs. This discussion is summarized in the table below. At the end of this section we summarize the proof of our main theorem and give an outline of the contents of this paper.

\begin{table}[h]
\centering
\begin{adjustbox}{width=\columnwidth,center}
\begin{tabular}{ | p{2cm} | p{4cm} | p{4cm} | p{4cm} | p{4cm} | p{4cm} |}
\hline
 & Euler characteristic &  Reidemeister torsion &Becker--Gottlieb transfer & higher torsion $\tau^{IK}$& smooth structure class $\Theta$ \\ \hline
Poincar\'e--Hopf theorem &

$\chi(M)=\displaystyle\sum_{z\in Z}(-1)^{\ind_X(z)}$ &
Compute torsion from explicit choice of cellular decomposition&
 Fiberwise Poincar\'e–Hopf theorem&
Framing Principle &
Theorem~\ref{thm:PHTheta}
\\ \hline
Poincar\'e duality &
$\ind_X(z)=-\ind_{-X}(z)$ for odd dimensional manifolds &
Recompute for dual chain complex&
Recompute fiberwise index map after negating Morse function& 
Apply the framing principle for a fiberwise GMF $f$ and compare with results for $-f$&
Duality Theorem --- Theorem~\ref{thm:duality}
\\ \hline

$\Rightarrow$Vanishing result  &
$\chi(M)=0$ when $\dim M$ is odd &
 Torsion is trivial on even dimensional manifolds&
 transfer map on real cohomology vanishes for odd dimensional fibers &
 higher torsion of even dimensional manifold bundles is MMM class&
Vanishing Theorem --- Theorem~\ref{thm:rigidity}
\\ \hline

\end{tabular}
\end{adjustbox}
\end{table}

 Recall the classical proof that the Euler characteristic of a closed odd dimensional manifold is zero. 
The Poincar\'e–Hopf theorem states that the Euler characteristic of a closed manifold $M$ is equal to the sum of the indices of isolated critical points of a Morse function $f$ on $M$:
$$
\chi(M)=\displaystyle\sum_{z\in Z}(-1)^{\ind_{\nabla f}(z)}
$$
On an odd dimensional manifold, ${\ind_{\nabla (-f)}(z)}$ and ${\ind_{\nabla f}(z)}$ have opposite parity.
It follows that $\chi(M)=0$. 

\smallskip

We adopt a stylized view of this proof: the vanishing result for the Euler characteristic is proven by applying the Poincar\'e–Hopf theorem in combination with Poincar\'e duality. 

\medskip

The Reidemeister torsion is a K-theoretic generalization of the Euler characteristic that admits an analogous vanishing theorem: the Reidemeister torsion of a closed even dimensional manifold is zero. To prove this, recall that the Reidemeister torsion of a manifold is computed in terms of an acyclic chain complex that can be obtained from a Morse function. We can compare the torsion of one chain complex to the torsion of the dual cell decomposition obtained by inverting the Morse function. The specific computation is due to Milnor~\cite{Mil62}, and when the dimension is even it follows that the torsion must be zero. Stylistically this proof is the same as the proof of the vanishing result for the Euler characteristic: the formula for torsion in terms of data obtained from the Morse function is an instance of a Poincar\'e–Hopf theorem. The comparison to the dual chain complex is an instance of Poincar\'e duality. A nearly identical argument is also used to prove a duality theorem for Whitehead torsion~\cite{Mil66}. 

\medskip

The Becker–Gottlieb transfer is a generalization of the Euler characteristic to families. For a smooth bundle of smooth manifolds $p:E\to B$, one can associate a wrong way map of spectra $\Sigma^\infty B_+ \to \Sigma^\infty E_+$. If the base is a point, this is equivalently a map of infinite loop spaces from $S^0$ to $\Omega^\infty\Sigma^\infty E_+$. On components,  if $E$ is connected, we have a map $S^0\to \Z$, and the non-basepoint element maps to $\chi(E)\in \Z$.

\smallskip

One can also prove a vanishing result for the Becker–Gottlieb transfer. Fiberwise Poincar\'e–Hopf theorems for the Becker–Gottlieb transfer have been proven by~\cite{BM76} and \cite{Dou06}. Briefly, assume that $X$ is a smooth nondegenerate vertical vector field on the total space of a smooth bundle $p:E\to B$. This vector field might be obtained by computing the gradient of a fiberwise Morse function, if such a function exists. Let $Z$ be the vanishing locus of the vector field, which forms a covering space $\pi:Z\to B$. Then the fiberwise Poincar\'e–Hopf theorem is expressed in terms of the following homotopy commutative diagram of spectra:

$$
\xymatrix{
\Sigma^\infty B_+ \ar[rr]^{{\tr}_p}\ar[dd]_{{\tr}_\pi}&&
\Sigma^\infty M_+\\
\\
 \Sigma^\infty Z_+
\ar[rr]^{{\ind}_X}&&
 \Sigma^\infty Z_+
\ar[uu]_{+}
}
$$

In the diagram above, $\tr_\pi$ and $\tr_p$ denote the transfers associated to $\pi$ and $p$. 
The map $\ind_X$ denotes a fiberwise index map associated to the vertical vector field $X$. 
On real cohomology one can prove that $(\ind_X)^* = (-1)^d(\ind_{-X})^*$, where $d$ denotes the fiber dimension of $M$. It follows that $(\tr_p)^* = (-1)^d (\tr_p)^*$. Thus the transfer map on cohomology vanishes when the fiber dimension is odd, if $p$ admits a smooth nondegenerate vertical vector field. In this case the classical Poincar\'e–Hopf theorem was replaced by a parametrized version, and Poincar\'e duality arose in the comparison of the vector field and its negative.

\medskip

A common generalization of the Euler characteristic to both the K-theoretic and parametrized settings is the higher Franz–Reidemeister torsion. This invariant is a characteristic class in the cohomology of the base of a smooth fiber bundle. The primary tool that enables computations of this invariant is Igusa's framing principle~\cite{Igu05}. The framing principle describes the higher Franz–Reidemeister torsion as the sum of an `exotic' class and a `tangential' term.

A consequence of the framing principle is that for smooth manifold bundles with closed even dimensional fibers, the torsion class is congruent to a Miller–Morita–Mumford class. To prove this, Igusa compares the framing principle for a fiberwise generalized Morse function $f$ to the analogous formula for $-f$. By studying the canonical involution on the Whitehead space, one can prove that the exotic term is two-torsion. Thus we are left only with the tangential term which agrees with a Miller–Morita–Mumford class. 

\smallskip

Once again, this proof is analogous to those from above: the framing principle resembles a Poincar\'e–Hopf theorem, and the comparison of the formulas for $f$ and $-f$ resembles an application of Poincar\'e duality. However, the proof of this vanishing theorem requires significantly more technology than those which came previously.  In particular, to define the exotic term in the framing principle, Igusa uses a Waldhausen category model for the Whitehead space which encodes the combinatorics of colliding critical points of fiberwise generalized Morse functions. The proof of the framing principle requires an understanding of the deformation properties of the critical loci of fiberwise generalized Morse functions. 

\medskip

We postpone giving a precise definition of the smooth structure class until Section~\ref{subsec:theta}, however we point out that the smooth structure class is closely related to the invariants discussed above. One explanation for this is that the higher Franz–Reidemeister torsion and the smooth structure class can both be defined in terms of nullhomotopies of maps that factor through the Becker–Gottlieb transfer. Given this relationship and the discussion above, one should expect that a proof of the vanishing theorem should follow from a sufficiently general version of a Poincar\'e–Hopf theorem along with an application of Poincar\'e duality. This paper provides such a proof, which is summarized in the next section.

\newpage
\subsection{Proof Summary}

In this paper we prove a vanishing result for the smooth structure class, an invariant of smooth structures on fiber bundles introduced by Goette, Igusa, and Williams in~\cite{GIW14}, after work of Dwyer, Weiss, and Williams~\cite{DWW03}. In analogy with the examples above, the proof is an application of a Poincar\'e–Hopf type theorem in combination with Poincar\'e duality. In this section we give an outline of the proof.

\medskip

By a fiberwise Poincar\'e–Hopf theorem we broadly mean a computation of a fiberwise characteristic, e.g. the Becker–Gottlieb transfer, the excisive A-theory Euler characteristic, etc., in terms of the critical locus of a fiberwise generalized Morse function. Examples of such theorems can be found in \cite{BM76,CJ98,Dou06}. These theorems generalize the classical Poincar\'e–Hopf theorem to the setting of smooth manifold bundles that admit fiberwise Morse functions. Unfortunately, bundles rarely admit fiberwise Morse functions, but by work of ~\cite{Igu84,Lur09,EM12}, smooth bundles always admit fiberwise generalized Morse functions. By working with such functions, Theorem~\ref{thm:PHQthy} generalizes existing fiberwise Poincar\'e–Hopf theorems to apply to all smooth manifold bundles. 

\medskip\noindent The following is a bullet-pointed outline of the proof of the main theorem. 

\begin{enumerate}\addtocounter{enumi}{-1}
    \item These background items are necessary for this outline: 
    \begin{itemize}
        \smallskip\item The smooth structure class $\Theta(W,M)$ is an element of $\pi_0\Gamma_B\h^\%_B(M)\otimes \Q$. The space $\Gamma_B\h_B^\%(M)$ is the space of sections of the fiberwise homology bundle obtained by taking fiberwise smash products with the stable h-cobordism space of a point. See Section~\ref{subsec:theta} for a precise definition. 
        \smallskip\item By the stable parametrized h-cobordism theorem, $\Gamma_B\h^\%_B(M)$ is the homotopy fiber of the map $\Gamma_B Q_B(M)\to \Gamma_B A^\%_B(M)$, which is induced by the unit map from the sphere spectrum to $A(*)$. The spectrum $A(*)$ is the Waldhausen K-theory of spaces functor, otherwise known as A-theory, evaluated at a point. The functor $A^\%$ is the excisive approximation to A-theory. 
        \smallskip\item All smooth bundles admit fiberwise generalized Morse functions by~\cite{Igu84,Lur09,EM12}. In stark contrast, smooth bundles rarely admit fiberwise Morse functions.
        \smallskip\item If a family of h-cobordisms $p:W\to B$ with boundaries $p_0:M_0\to B$ and $p_1:M_1\to B$ is \textit{topologically trivial}, then $W$ is fiberwise tangentially homeomorphic to $M_0\times I$. This data produces a nullhomotopy of the excisive A-theory Euler characteristic of $\chi^\%(W,\partial_0W)$, which is used to define the smooth structure characteristic $\theta(W,\partial_0W)$ in Definition \ref{defn:theta}. Likewise, we define the smooth structure characteristic $\theta(W,\partial_1W)$.
    \end{itemize}
    \smallskip\item We prove a fiberwise Poincar\'e–Hopf theorem for the Becker–Gottlieb transfer, an element of $\Gamma_B Q_B(W)$. This result, Theorem~\ref{thm:PHQthy}, expresses the transfer in terms of the critical locus of a fiberwise generalized Morse function on $W$. 
    \smallskip\item In Theorem~\ref{thm:PHAthy}, we further generalize the fiberwise Poincar\'e–Hopf theorem for the Becker–Gottlieb transfer to the excisive A-theory Euler characteristic, an element of $\Gamma_B A^\%_B(W)$. Furthermore, this factorization is compatible with the factorization of the Becker–Gottlieb transfer in the previous theorem. This is Theorem~\ref{thm:PHcube}.
    \smallskip\item We generalize Theorems \ref{thm:PHQthy}, \ref{thm:PHAthy}, and \ref{thm:PHcube} to be invariant under any \emph{stratified deformation} of the critical locus of a fiberwise generalized Morse function. This yields Theorems \ref{thm:PHHSD} and \ref{thm:PHcubeSD}.
    \smallskip\item Theorem~\ref{thm:PHTheta} gives a fiberwise Poincar\'e–Hopf theorem for the smooth structure characteristic. This theorem is different from those preceding it in that it gives a congruence between elements of a rational group as opposed to proving the homotopy commutativity of a diagram of commutative maps. This theorem relies on Theorem~\ref{thm:PHHSD} to express the smooth structure class in terms of a particular stratified deformation of the critical locus of a fiberwise generalized Morse function. The stratified deformation that we use encodes a parametrized handle cancellation argument used by Hatcher in~\cite{Hat75} and Igusa in~\cite{Igu84,Igu88,Igu02,Igu05}. We summarize this construction in Subsection \ref{subsec:strat_def}.
    \smallskip\item Theorem~\ref{thm:PHTheta} is used to prove a duality theorem for the smooth structure class, Theorem~\ref{thm:duality}, by inverting the Morse function. Theorem~\ref{thm:rigidity} follows from this duality theorem. 
\end{enumerate}

To complete the analogy in the exposition from the previous section, we will indicate how the proof of the vanishing theorem, Theorem~\ref{thm:rigidity}, can be thought of as an application of a Poincar\'e–Hopf type theorem and Poincar\'e duality. The Poincar\'e–Hopf type theorem that we ultimately apply is Theorem~\ref{thm:PHTheta}, and as the outline indicates, this is a generalization of other Poincar\'e–Hopf theorems that we prove along the way. The duality theorem for the smooth structure class, Theorem~\ref{thm:duality}, is an instance of Poincar\'e duality.

\subsection{Acknowledgements}

I wish to express my gratitude to John Francis for his patience and support and for many helpful conversations about this project. I owe a tremendous intellectual debt to Kiyoshi Igusa who has generously shared his ideas with me and whose work has provided much inspiration for my own. I also want to thank Oscar Randal-Williams for many helpful conversations, and especially for pointing out mistakes, conceptual and otherwise, in earlier versions of this work.

\smallskip

Special thanks to Sam Nariman for the many hours spent discussing manifold topology and early ideas about this project, and to Piotr Pstr{\c a}gowski for his support and for reading early drafts of this work. At various stages this project also benefitted from conversations with Mauricio Bustamante, Mannuel Krannich, and Wolfgang Steimle. I also wish to thank Chris Douglas and John Klein for helpful email correspondences.

\section{Related Work}
\label{sec:RelatedWork}

\subsection{Comparison to prior work of Goette, Igusa, and Williams}

\label{subsec:comparison_to_GIW}

This work grew out of the author's attempts to prove the Rigidity Conjecture of Goette and Igusa. A careful reader will notice that the statement of the vanishing theorem given here is different from the statement of the conjecture given in~\cite{GI14}. Furthermore, the papers~\cite{GIW14} and~\cite{GI14} are about exotic smoothings of manifold bundles, and do not make explicit use of topologically trivial families of smooth h-cobordisms. We will explain these differences below. 

In~\cite{GIW14}, Goette, Igusa, and Williams define an exotic smoothing of a manifold bundle $p:M\to B$ to be another manifold bundle $p':M'\to B$ along with a \textit{fiberwise tangential homeomorphism} from $M$ to $M'$. A fiberwise tangential homeomorphism is a fiberwise homeomorphism $h$ covered by an isomorphism of the vertical tangent bundles that is compatible with the topological derivative of $h$. They construct the moduli space of exotic smoothings of $p$, $\widetilde{\s}_B(M)$, and its stabilization with respect to linear disk bundles $\widetilde{\s}^s_B(M)$. The main smoothing theorem of their paper is a homotopy equivalence
\[
\widetilde{\s}^s_B(M)\simeq\Gamma_B\h_B^\%(M)
\]
between the stable space of exotic smoothings of $M$ and the sections of the parameterized ring-spectrum obtained by taking a fiberwise smash product of $M$ with the stable parameterized h-cobordism space of a point. 

In this paper, Corollary \ref{cor:DWWpullback} gives a homotopy equivalence between the space of topologically trivial families of smooth h-cobordisms on $M$, denoted $\h_B^{t/d}(M)$, and $\Gamma_B\h_B^\%(M)$. Thus, there is a homotopy equivalence between the spaces $\h_B^{t/d}(M)$ and $\widetilde{\s}^s_B(M)$. In a subsequent paper we will construct an explicit homotopy equivalence $\h_B^{t/d}(M)\to \widetilde{\s}^s_B(M)$. This homotopy equivalence allows us to interchange topologically trivial familes of smooth h-cobordisms with exotic smoothings of manifold bundles. 

\medskip

In~\cite{GI14}, the authors provide a construction called the \textit{immersed Hatcher construction}. 
From the perspective used in our paper, their construction uses Hatcher's construction of exotic disk bundles to produce topologically trivial families of smooth h-cobordisms on an arbitrary smooth manifold bundle $p:M\to B$. 
In~\cite{GI14}, the authors claim that the immersed Hatcher construction produces ``exotic smooth structures on smooth manifold bundles with closed fibers." 
As was pointed out by Oscar Randal-Williams to the author of this paper, exotic smooth structures on smooth manifold bundles with closed fibers do not exist. 
Indeed, if there is a fiberwise tangential homeomorphism between manifold bundles with closed fibers, then by smoothing theory these bundles are fiberwise diffeomorphic.
However, \textit{stable} exotic smooth structures do exist on manifold bundles with closed fibers, e.g. if $W$ is a topologically trivial family of smooth h-cobordisms from $M$ to $M'$ then $M\times I$ is an exotic smoothing of $M'\times I$. 

Unfortunately, the false assertion that exotic smoothings of manifold bundles with closed fibers exist causes the statements of some results in~\cite{GI14} to be false, and affects the proofs of other statements. These specific errors were also uncovered by Oscar Randal-Williams. 

Of particular concern to us is the statement of the Rigidity Conjecture in~\cite{GI14}, which is vacuously true. 
In the original statement the authors are comparing two smooth manifold bundles $p:M\to B$ and $p':M'\to B$ with closed even dimensional fibers so that $p'$ is an exotic smoothing of $p$, i.e. $M'$ is fiberwise tangentially homeomorphic to $M$. 
The conjecture asserts that the smooth structure class comparing the two bundles vanishes.
As Randal-Williams has pointed out, with the conditions above $M$ must be fiberwise diffeomorphic to $M'$, and so the smooth structure class must vanish. 

The statement of the vanishing theorem given in this paper is the most reasonable restatement of the Rigidity Conjecture of Goette and Igusa. 
It compares two manifold bundles $p:M\to B$ and $p':M'\to B$ with closed even dimensional fibers that arise as the boundaries of a topologically trivial family of smooth h-cobordisms. 
Given this setup, $M\times I$ is fiberwise tangentially homeomorphic to $M'\times I$, and so $M'\times I$ is an exotic smoothing of $M\times I$. 
Furthermore, we know from the homotopy equivalence above that all exotic smoothings arise in this way.
Now the two bundles $M\times I$ and $M'\times I$ are not necessarily fiberwise diffeomorphic, and so the smooth structure class is not obviously zero.

We wish to point out that the errors in~\cite{GI14} are contained within that publication and do not affect the results in~\cite{GIW14}. Furthermore the geometric constructions in~\cite{GI14} are sound, and the errors in that paper only pertain to the interpretation of those constructions and the computation of related invariants. Finally, the present paper is mathematically independent of ~\cite{GI14}. In a subsequent paper we will supply corrected proofs of the main theorems of~\cite{GI14}, using the perspective of topologically trivial families of smooth h-cobordisms.

\subsection{A survey of nearby literature}

Two main ideas arising in this paper, fiberwise Poincar\'e–Hopf theorems and the Hatcher construction, are of independent interest. In this section we give a brief survey of the literature and recent progress pertaining to both keywords. We omit a discussion of the most immediate literature pertaining to the Rigidity Conjecture, including~\cite{DWW03, GIW14, GI14}, as detailed descriptions of these works appear elsewhere in this paper.

\medskip

Hatcher's construction associates to an element of the kernel of the J-homomorphism a disk bundle which is fiberwise homeomorphic to a trivial disk bundle but not fiberwise diffeomorphic. The construction gives rise to a stable map from G/O to $\Omega\wh^\diff(*)$. Waldhausen gave a different formulation of the same map in~\cite{Wal82}, and in~\cite{Bok84} Bokstedt proved that this map is a rational homotopy equivalence. Exciting new developments by Kragh~\cite{Kra18} have identified the homotopy fiber of the Hatcher--Waldhausen map as a certain functional space $\mathcal{M}_\infty$ considered by Eliashberg and Gromov in~\cite{EG98}, establishing a connection between the study of Lagrangian submanifolds to algebraic K-theory. Kragh associates to every exact Lagrangian of $T^*D^n$ which agrees with the zero section at $\partial D^n$ an element of $\pi_*(\mathcal{M}_\infty)$. If any of these examples were proven to be nontrivial, they would be counterexamples to the nearby Lagrangian conjecture in symplectic geometry. In essence, counter examples to the nearby Lagrangian conjecture might be found in the kernel of Hatcher's construction. Related work by Igusa and \'Alvarez-Gavela produces torsion invariants for Legendrian submanifolds of Euler type~\cite{GI19}.

\medskip

Goodwillie, Igusa, and Ohrt have developed an equivariant version of Hatcher's construction~\cite{GIO15}. Ordinarily, the space $G/O$ classifies vector bundles whose spherical fibrations are fiber homotopy trivial. In the equivariant version, $G/O$ is replaced by the space $G_n/U$, which is the colimit of spaces $G_n(N)/U(N)$, classifying rank $N$ complex vector bundles together with a $C_n$-equivariant fiber homotopy trivialization of the associated sphere bundle. The equivariant Hatcher construction is then a map $G_n/U\to \h^s(BC_n)$, where the target is the stable h-cobordism space of the classifying space of $C_n$. The geometric outcome of the construction is no longer a disk bundle, but instead a bundle of h-cobordisms on the product of a disk with lens spaces.

\medskip

Bunke and Gepner have reformulated the Becker--Gottlieb transfer in the context of derived algebraic K-theory~\cite{BG13}. Their work leads to the Transfer Index Conjecture, essentially a derived version of the parametrized index theorem of Dwyer, Weiss, and Williams. This conjecture suggests as a corollary the existence of certain classes in algebraic K-theory of a ring of integers in a number field. The authors prove that the Hatcher construction produces nontrivial representatives for these classes in special cases.

\medskip

The Farell--Hsiang~\cite{FH78} results on diffeomorphism groups of disks relative to their boundary prove that in the pseudoisotopy stable range,

\begin{equation}
\pi_i\bdiff_\partial(D^n)\otimes\Q=\left\{
\begin{aligned}
    \Q \ && \text{if $i=0\mod{4}$ and $n$ odd}\\ 
     0 \ && \text{otherwise}
\end{aligned}
\right.
\end{equation}
According to~\cite{Igu02}, the nontrivial generator for odd dimensional disks can be obtained from Hatcher's construction. There has been much recent progress outside of the stable range by a number of authors, including distinct but related work by Kupers, Randal-Williams, Watanabe, and Weiss. In~\cite{Wei15} Weiss identifies nontrivial Pontryagin classes $p_{n+k}\in H^{4n+4k}(\btop(2n);\Q)$. These are shown to evaluate nontrivially on $\pi_{4n+4k}(\btop(2n);\Q)$. It then follows by the Morlet equivalence that there must be nonzero rational homotopy in $\pi_*(\bdiff_\partial(D^{2n}))$ outside of the stable range. It would be interesting to determine whether these classes could be used to produce unstable exotic smoothings of manifold bundles. 

\medskip

Fiberwise Poincar\'e--Hopf theorems first appeared in work by Brumfiel and Madsen~\cite{BM76}, and have proven to be a useful computational tool with many applications. For instance, in~\cite{MT01} the authors apply the theorem towards early progress on the Mumford conjecture. Similar theorems are used in~\cite{ORW08} to compute the $\Z/2$ homology of the stable nonorientable mapping class group, as well as in~\cite{Rei19} to establish the existence of non-kinetic smooth bundles over the classifying space $BSU(2)$. Douglas~\cite{Dou06} gave alternative proofs to ~\cite{BM76} using Dold's Euclidean neighborhood rectracts. Both of these papers establish fiberwise Poincar\'e–Hopf theorems for smooth manifold bundles admitting fiberwise Morse functions, and provided motivation for the present work. Our results provide fiberwise Poincar\'e--Hopf theorems for bundles that admit fiberwise generalized Morse functions. This is a significant strengthening, as all bundles admit fiberwise generalized Morse functions, but bundles rarely admit fiberwise Morse functions.  

\medskip

This paper is concerned with a characteristic in the homotopy fiber of the fibration

$$\Gamma_B\h^\%_B(M)\to \Gamma_BQ_B(M)\to \Gamma_B A^\%_B(M)$$

\noindent that arises from a nullhomotopy of the excisive A-theory Euler characteristic. One could naturally ask about characteristics in the homotopy fiber of the fibration

$$\Gamma_B\Omega\wh^\text{PL}_B(M)\to \Gamma_BA^\%_B(M)\to \Gamma_B A_B(M)$$

\noindent given by a nullhomotopy of the ordinary A-theory characteristic. This is the premise of~\cite{Ste10}, in which the author studies the `parametrized excisive characteristic'. One of the main technical results of their work is an additivity theorem for the parametrized excisive characteristic, which is related to the Poincar\'e--Hopf theorems in the present work when they are restricted to bundles admitting fiberwise Morse functions. For comparison, the smooth structure class appearing in $\Gamma_B\h^\%_B(M)$ concerns the existence of stable exotic smoothings of fiber bundles, whereas the parametrized excisive characteristic appearing in $\Gamma_B\Omega\wh^\text{PL}_B(M)$ concerns the existence of topological manifold bundles whose projection maps are homotopic to stabilizations of a map of compact topological manifolds.

\section{Characteristics associated with topologically trivial families of h-cobordisms}
\label{sec:2}

In this section we formally introduce topologically trivial families of smooth h-cobordisms, with the goal of defining the smooth structure characteristic. 
The content of this section is used in Subsection~\ref{subsec:PHTheta} and Section~\ref{sec:ThetaCalc}, where we apply the theorems of Section~\ref{sec:PHthms} to topologically trivial families of smooth h-cobordisms.
We begin in Subsection~\ref{subsec:top_triv_families} with the definition of a topologically trivial family of smooth h-cobordisms, and introduce the immersed Hatcher construction as the main example of such a bundle. In Subsection~\ref{subsec:moduli_hcobs}, we identify the moduli space of topologically trivial h-cobordisms as the homotopy fiber of the forgetful map from the space of smooth h-cobordisms to the space of topological h-cobordisms. In Subsection~\ref{subsec:char_hcobs} we introduce the smooth structure characteristic, a characteristic of topologically trivial families of smooth h-cobordisms defined as a lift of the Becker–Gottlieb transfer.

\subsection{Topologically trivial families of h-cobordisms} 
\label{subsec:top_triv_families}

In this subsection we define topologically trivial families of smooth h-cobordisms and provide examples of such objects. 

\begin{defn}
A smooth family of h-cobordisms $p:W\to B$ with boundaries $\partial_0W:= M$ and $\partial_1W:=M'$ given as smooth manifold bundles $p_0:M\to B$ and $p_1:M'\to B$ is topologically trivial if there exists a fiberwise homeomorphism $h:W\to M\times I$ over $B$. 
\end{defn}

\begin{rmk}
When the base is a point, the Whitehead torsion of a topologically trivial h-cobordism must be trivial, and thus by the s-cobordism theorem $W$ must be a cylinder. This is not necessarily the case when the base is not contractible. In particular, there may be interesting maps from the base into the basepoint component of the loopspace of the smooth Whitehead space. 
\end{rmk}

\begin{example}
Hatcher's construction takes as input a vector bundle classified by $G/O$ and produces disk bundles that are fiberwise homeomorphic to a trivial disk bundle, but not fiberwise diffeomorphic. The immersed Hatcher construction~\cite{GI14} utilizes Hatcher's disk bundles to produce topologically trivial h-cobordisms. Briefly, given a smooth manifold bundle $p_0:M\to B$, they consider the bundle $p_0\times I: M\times I\to B$ and glue a family of cancelling handles parametrized by $B$ on the outgoing boundary $M\times 1\to B$. This family of handles is constructed to be one of Hatcher's disk bundles. The result of this construction is a topologically trivial family of smooth h-cobordisms, as the bundle remains fiberwise homeomorphic to $M\times I\to B$. This family of smooth h-cobordisms is not necessarily fiberwise diffeomorphic to $M\times I\to B$, as the smooth structure on the outgoing boundary has changed according to the smooth structure on Hatcher's disk bundle. See~\cite{GI14} for details of the construction. 
\end{example}

\subsection{Moduli spaces of h-cobordisms}
\label{subsec:moduli_hcobs}

Given a smooth manifold $F$, we now define the space of smooth h-cobordisms on $F$, the space of topological h-cobordisms on $F$, and the space of topologically trivial smooth h-cobordisms on $F$. We also introduce notation for the stable versions of these spaces. 

\begin{defn}
Let $H^{t}(F)$ denote the space of topological h-cobordisms on $F$. 
This space is defined to be the geometric realization of a simplicial set $H_\bullet^{t}(F)$. 
A $k$-simplex in $H_\bullet^{t}(F)$ is a  topological manifold bundle $\pi:E\to\Delta^q$ for which each fiber $W_p = \pi^{-1}(p)$ is a topological h-cobordism on $F$. 
We denote by $H^{t}_B(F)$ the mapping space $|H_\bullet^{t}(F)|^B$.
\end{defn}

\begin{defn}
Let $H^{d}(F)$ denote the space of smooth h-cobordisms on $F$. 
This space is defined to be the geometric realization of a simplicial set $H_\bullet^{d}(F)$. 
A $k$-simplex in $H_\bullet^{d}(F)$ is a  smooth bundle $\pi:E\to\Delta^q$ for whcih each fiber $W_p = \pi^{-1}(p)$ is a smooth h-cobordism on $F$. 
We denote by $H^{d}_B(F)$ the mapping space $|H_\bullet^{d}(F)|^B$.
\end{defn}

\begin{defn}
Let $H^{t/d}(F)$ denote the space of topologically trivial h-cobordisms on $F$. 
This space is defined to be the geometric realization of a simplicial set $H_\bullet^{t/d}(F)$. 
A $k$-simplex in $H_\bullet^{t/d}(F)$ is a pair $(\pi,h)$ for which the map $\pi:E\to\Delta^q$ is a smooth bundle such that each fiber $W_p = \pi^{-1}(p)$ is a smooth h-cobordism on $F$. 
The map $h$ is a homeomorphism from $E$ to $F\times I\times \Delta^k$ over $\Delta^k$. 
We denote by $H^{t/d}_B(F)$ the mapping space $|H_\bullet^{t/d}(F)|^B$.
\end{defn}

Let $\h^X_B(F)$ denote the stabilizations of the spaces $H^X_B(F)$ with respect to stabilization maps $H^X_B(F)\to H^X_B(F\times I)$ for $X$ being any of $t,d,\text{ or }t/d$. 

\begin{prop}
The space $\h^{t/d}(F)$ is the homotopy fiber of the forgetful map $\h^d(F)\to \h^t(F)$ over $F\times I$. 
\end{prop}

\begin{proof}
It suffices to see that the unstable space $H^{t/d}(F)$ is the homotopy fiber of the forgetful map $H^d(F)\to H^t(F)$. Let $(\pi,h)$ be a zero simplex in $H_\bullet^{t/d}(F)$. That is, $\pi:E\to *$ is a smooth manifold bundle with $E$ is a smooth h-cobordism on $F$, and $h$ is a homeomorphism from $F\times I$ to $E$. Then $\pi$ is clearly a point in $H^d(F)$. It remains to show that $h$ is equivalent to a 1-simplex in $H_\bullet^t(F)$. Thus from $h$ we must obtain a one parameter family of topological h-cobordisms on $F$ that starts at $F\times I$ and ends at $E$. Consider the h-cobordism $F\times I\cup_{\partial_0} E$ in $H^d(F)$, which is diffeomorphic to $E$, and homeomorphic to $F\times [-1,1]$. Let $\pi': E\cup F\times I\to [-1,1]$ be the new projection map for this  family of topological h-cobordisms, and consider the h-cobordisms given by $\pi'^{-1}[-t,1]$ for $t\in[0,1]$. This is the desired one parameter family of topological h-cobordisms.  
\end{proof}

\subsection{Characteristics of h-cobordisms and the Dwyer–Weiss–Williams pullback square}
\label{subsec:char_hcobs}

In this section we introduce the relative Becker–Gottlieb transfer and relative excisive A-theory Euler characteristic on families of smooth and topological h-cobordisms, respectively. 
We then recall a result from~\cite{DWW03} which situates these characteristics in a homotopy pullback square. 

Associated to a smooth manifold bundle $p:M\to B$ with compact fibers, our characteristics will be points in the section spaces $\Gamma_BQ_B(M_+)$ and $\Gamma_BA_B^\%(M)$. 
Roughly speaking, these spaces are sections of the fiberwise homology bundles obtained by taking a fiberwise smash product with the sphere spectrum $\S$ and the algebraic K-theory of spaces functor evaluated at a point, $A(*)$. Moreover, these section spaces are related by a map $\eta:\Gamma_BQ_B(M_+)\to\Gamma_BA_B^\%(M)$ induced by the unit map $\S\to A(*)$.

The Becker–Gottlieb transfer is a section $\tr(p)\in\Gamma_BQ_B(M_+)$, for which the composition of $\tr(p)$ and the inclusion map $Q_B(M_+)\hookrightarrow Q(M_+)$ is the usual transfer $B\to Q(M_+)$. The excisive A-theory Euler characteristic is a section $\chi^\%(p)\in\Gamma_BA_B^\%(M)$. More precise homotopical formulations of these spaces and exact definitions of these characteristics are given in Section~\ref{sec:PHthms}.

Given that our purpose is to study h-cobordisms, we also require \textit{relative} versions of these characteristics. When considering a smooth h-cobordism bundle $p:W\to B$ with boundaries $\partial_0W:= M$ and $\partial_1W:=M'$ given as smooth manifold bundles $p_0:M\to B$ and $p_1:M'\to B$, we require versions of the Becker–Gottlieb transfer and excisive A-theory Euler characteristics that are relative to $\partial_0W$.

We denote by $\tr_\partial(p)$ the section $\tr(p)-r^*\tr(p_0)$ in $\Gamma_BQ_B(W)$ where $r$ is the retraction of $W$ onto $M$. We denote by $\chi^\%_\partial(p)$ the section $\chi^\%(p)-r^*\chi^\%(p_0)$ in $\Gamma_BA_B^\%(W)$. 

\medskip

\noindent Dwyer, Weiss, and Williams proved the following theorem relating $\tr(p)$ and $\chi^\%(p)$.
\begin{numberedtheorem}[Index Theorem~\cite{DWW03}]
For $p:M\to B$ a bundle of compact smooth manifolds, $\chi^\%(p)\in\Gamma_BA_B^\%(M)$ is fiberwise homotopic to $\eta\circ\tr(p)$.
\end{numberedtheorem}

\noindent The theorem above implies the commutativity of the diagram in the following stronger theorem also proven by Dwyer, Weiss, and Williams.

\begin{numberedtheorem}[Corollary 12.3 in~\cite{DWW03}]\label{thm:DWWpullback}
The following diagram is a homotopy pullback square: 
\begin{equation}\label{diag:DWWpullback}
\xymatrix{
\h^d_B(F)\ar[r]^-{\tr_\partial(p)}\ar[d]_{\text{forget}}& \Gamma_BQ_B(M_+)\ar[d]^\eta\\
\h^t_B(F)\ar[r]^-{\chi^\%_\partial(p)}& \Gamma_B A^\%_B(M)
}
\end{equation}
\end{numberedtheorem}

\begin{rmk}
This theorem is an essential ingredient in the proof of Dwyer, Weiss, and William's converse Riemann–Roch theorem.
\end{rmk}

\begin{corollary}\label{cor:DWWpullback}
The homotopy fiber of the left vertical arrow in diagram~(\ref{diag:DWWpullback}) is the space $\h^{t/d}_B(F)$, and the homotopy fiber of the right vertical arrow is the space $\Gamma_B\h^\%_B(M)$, the space of sections of the space obtained by taking a fiberwise smash product with $\h(*)$. 
The induced map on these homotopy fibers,
$\theta:\h_B^{t/d}(F)\to \Gamma_B\h_B^\%(M)$ 
is a homotopy equivalence. 
\end{corollary}

\subsection{The smooth structure characteristic}
\label{subsec:theta}

In light of Corollary~\ref{cor:DWWpullback}, we will now define the smooth structure characteristic for a single family of topologically trivial h-cobordisms. Associated to a topologically trivial family of h-cobordisms is a canonical nullhomotopy of the excisive A-theory Euler characteristic. This nullhomotopy is used to define the smooth structure characteristic, as in the following definition.

\begin{defn}\label{defn:theta}
The smooth structure characteristic of a topologically trivial family of h-cobordisms $p:W\to B$, denoted $\theta(W,\partial_0W)$, is a section in $\Gamma_B\h_B^\%(W)$ canonically determined by the point $\tr_\partial(p)\in \Gamma_B Q_B(W)$ over $\chi^\%_\partial(p)\in\Gamma_B A_B^\%(W)$, and the path from $\chi^\%_\partial(p)$ to $\chi^\%(M)$ determined by the fiberwise homeomorphism $h:W\to M\times I$. 
\end{defn}

\begin{rmk}
In keeping with the definition of higher smooth torsion due to~\cite{DWW03} as a nullhomotopy of the composition of the Becker–Gottlieb transfer with the map $Q(M_+)\to K(\Z)$, Definition~\ref{defn:theta} presents the smooth structure characteristic as a nullhomotopy of the excisive A-theory Euler characteristic. Thus, we may think of the smooth structure characteristic as a refinement of the higher smooth torsion.
\end{rmk}

\section{Review of generalized Morse functions}
\label{sec:GMFs}

In this section we introduce the theory of generalized Morse functions used in the rest of this paper. In Subsection~\ref{subsec:gmf_trans} we define generalized Morse functions, give the local properties of these functions as they vary in families, and prove the transversality result that implies that the critical locus of a fiberwise generalized Morse function is a submanifold of the total space. In Subsection~\ref{subsec:ghosts} we introduce \textit{ghost sets}, a perturbation of the critical locus of a fiberwise generalized Morse function in the neighborhood of a birth-death singularity which will be used in subsequent sections to compute characteristics associated to the critical locus. In Subsection~\ref{subsec:strat_def} we introduce stratified subsets as a generalization of the critical locus of a fiberwise generalized Morse function. We define stratified deformations of stratified subsets, and give a general purpose construction that deforms the critical locus into a stratified subset that lies in two consecutive degrees. This property is used in Subsection~\ref{subsec:PHTheta} to prove a fiberwise Poincar\'e–Hopf theorem for the smooth structure characteristic, and in Subsection~\ref{subsec:duality_thm} to prove a duality theorem for the smooth structure characteristic.

\subsection{Definitions and transversality properties of generalized Morse functions}
\label{subsec:gmf_trans} 

We begin with the definition of a \textit{generalized Morse function}.

\begin{defn}
A generalized Morse function on a single manifold $M$ is a function $f:(M,\partial_0M)\to (I,0)$ that admits only Morse and birth-death critical points. 
\end{defn}

\noindent In local coordinates, a Morse critical point of the function $f$ can be written in the form
$$
f(x)=-x_1^2-\cdots-x_i^2+x_{i+1}^2+\cdots+x_n^2
$$
with respect to coordinates $(x_1,\cdots,x_n)\in\R^i\times\R^{n-i}$. At a birth-death critical point, the normal form is as follows: 
$$
f(x)=-x_1^2-\cdots-x_{i-1}^2+x_i^3+x_{i+1}^2+\cdots+x_n^2
$$

\medskip

In this paper, we are interested in families of generalized Morse functons. Let $f:(W,\partial_0W)\to (I,0)$ be a fiberwise generalized Morse function, meaning that its restriction to each fiber is a generalized Morse function. Igusa proved that such functions always exist on smooth fiber bundles when the dimension of the fiber is at least the dimension of the base~\cite{Igu90}, and this dimensionality condition was later relaxed by independent work of Lurie~\cite{Lur09} and Eliashberg–Misachev~\cite{EM12}. Our goal is to recall the local behavior of such functions, and to illustrate the key transversality property enjoyed by their critical loci.

\medskip\noindent In the parametrized setting, we have the following proposition/summary from~\cite{Igu05}:

\begin{prop}\label{prop:gmf_normal_form}
In a generic $p$-parameter family of generalized Morse functions, birth-death points occur on a codimension one subspace of the parameter space. 
The family of functions $f_t$ has the form
$$
f_t(x)=-x_1^2-\cdots-x_{i-1}^2+x_i^3+t_0x_i+x_{i+1}^2+\cdots+x_n^2
$$
with respect to parameter coordinates $t_0,\cdots t_{p-1}$ and $t$-dependent local coordinates $(x_1,\cdots,x_n)$ for $M$.
\end{prop}

The coordinate $t_0$ in the proposition above is often referred to as the `unfolding directon' associated to the birth-death critical point. This suggests the following useful depiction of a birth-death critical point, which might be interpreted as a `cancellation' of Morse critical points, or their associated handles. 

\begin{center}
\begin{figure}[h]
\begin{overpic}[abs,width=8cm]{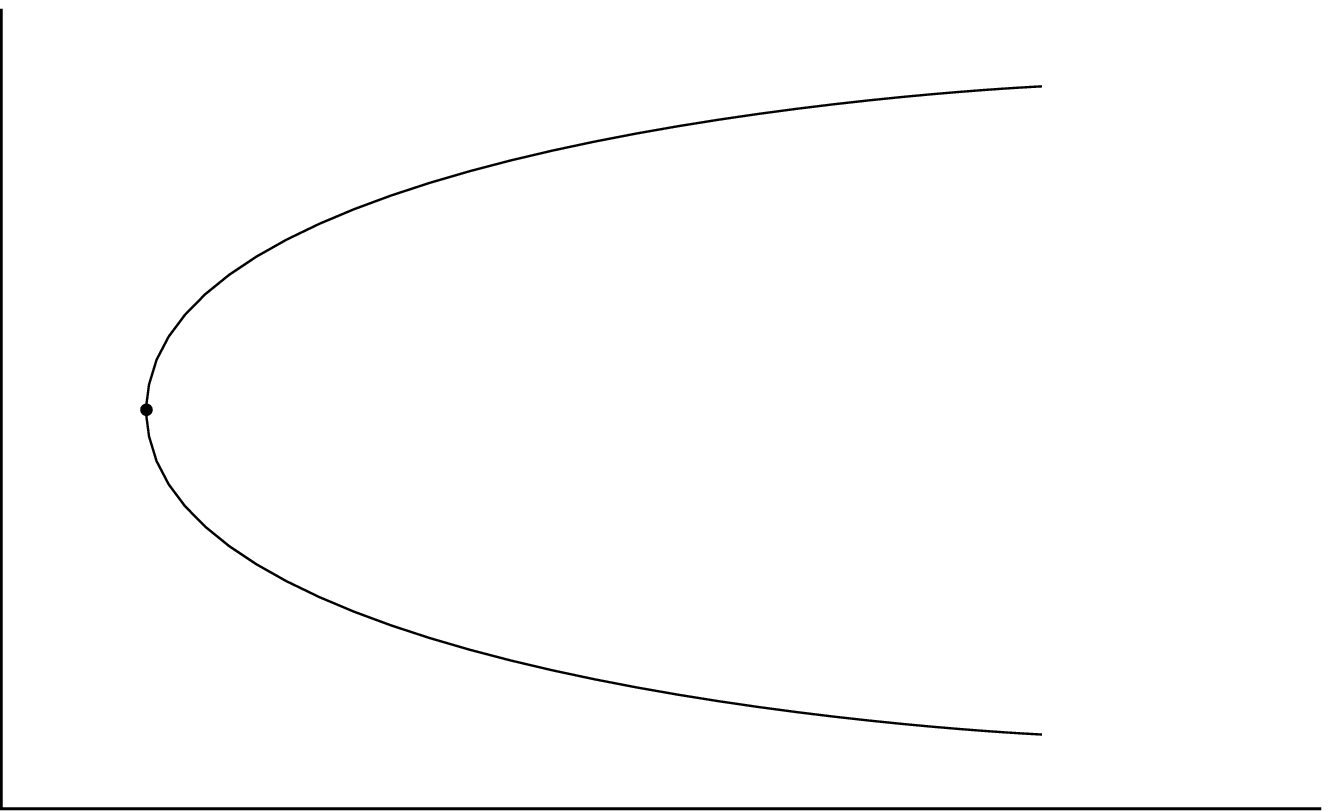}
\put(-10,70){$x$}
\put(110,-10){$t_0$}
\end{overpic}
\caption{A depiction of the critical locus of a fiberwise generalized Morse function in the local neighborhood of a birth-death critical point.}
\end{figure}
\end{center}

To obtain our desired transversality result, we proceed to compute the Hessian of $f$ at a birth-death critical point.

\medskip

Let $p:W^m\to B^k$ be a smooth fiber bundle with compact base and fiber $F^n$. 
Let $f:(W,\partial_0W)\to (I,0)$ be a fiberwise generalized Morse function as before. 
Then by the proposition above, in local coordinates at a birth-death singularity we have
$$
f_t(x) = -x_1^2-\cdots-x_i^2+x_{i+1}^3+t_0x_{i+1}+x_{i+2}^2+\cdots+x_n^2.
$$
where $t_0$ is the unfolding direction. 
The gradient of this function is a section $W\to TW$, and if we take the gradient with respect to fiber coordinates, we have a section of the vertical tangent bundle, $W\to T^\vee W$.
We can explicitly compute the map $\nabla f:W\to T^\vee W$ as 
$$
(x,t)\mapsto \langle -2x_1,\cdots, -2x_i,3x_{i+1}^2+t_0,2x_{i+2},\cdots,2x_n\rangle
$$ 
The derivative of $\nabla f$ is a map on tangent spaces with the last map in the composition below being the projection off of the nonidentity component. 
$$
T_{(x,t)}W\mapsto T_{\nabla f(x,t)}(T^\vee W)\cong T^\vee W\oplus T^\vee W\rightarrow T^\vee W
$$

This map takes the form of a rectangular matrix of size $(2n+k)\times (n+k)$, which is written below. 
The tangent space in the domain is labeled using coordinates $t_0,\cdots, t_{k-1}$ in the base, and $x_1,\cdots,x_n$ in the fiber. 
In the target we add labels $\frac{\partial}{\partial x_i}$ for $i,\cdots, n$ for the coordinates in the vertical tangent direction. 
Keep in mind that in the neighborhood of a birth-death singularity, $t_0$ is always identified with the `unfolding` direction.

\medskip

\begin{adjustbox}{width=.6\columnwidth,center}
\vphantom{
    $
    \begin{matrix}
    \overbrace{XYZ}^{\mbox{$R$}}\\ \\ \\ \\ \\ \\
    \underbrace{pqr}_{\mbox{$S$}}
    \end{matrix}
    $
    }%
    $
\begin{matrix}
    \coolleftbrace{x}{e \\ y\\ y\\e \\ y\\ y\\y}\\
    \coolleftbrace{t}{e \\ y\\ y\\e \\ y\\ y\\y}\\
    \coolleftbrace{\frac{\partial}{\partial x}}{e \\ y\\ y\\e \\ y\\ y\\y \\y}
\end{matrix}%
    $
    $
\left(
\begin{array}{ccccccc:ccccccc}
\coolover{x}{\phantom{-}1\phantom{-} & \phantom{\ddots} &  \phantom{0} & \phantom{6x_i} &  \phantom{0} & \phantom{\ddots} & \phantom{0}  &}  \coolover{t}{ \phantom{0} &  \phantom{0} & \phantom{0}  & \phantom{0}  & \phantom{0}  & \phantom{0}  &  \phantom{0}}\\
      & 1 &   &   &   &   &   &   &   &   &   &   &   &  \\
      &   & 1 &   &   &   &   &   &   &   &   &   &   &  \\
      &   &   & 1 &   &   &   &   &   &   &   &   &   &  \\
      &   &   &   & 1 &   &   &   &   &   &   &   &   &  \\
      &   &   &   &   & 1 &   &   &   &   &   &   &   &  \\
      &   &   &   &   &   & 1 &   &   &   &   &   &   &  \\
    \hdashline
      &   &   &   &   &   &   & 1 &   &   &   &   &   &  \\
      &   &   &   &   &   &   &   & 1 &   &   &   &   &  \\
      &   &   &   &   &   &   &   &   & 1 &   &   &   &  \\
      &   &   &   &   &   &   &   &   &   & 1 &   &   &  \\
      &   &   &   &   &   &   &   &   &   &   & 1 &   &  \\
      &   &   &   &   &   &   &   &   &   &   &   & 1 &  \\
      &   &   &   &   &   &   &   &   &   &   &   &   & 1\\
    \hdashline
    -2 &   &   &   &   &   &   &   &   &   &   &   &   &  \\
      & \ddots &   &   &   &   &   &   &   &   &   &   &   &  \\
      &   & -2 &   &   &   &   &   &   &   &   &   &   &  \\
      &   &   & 6x_{i+1} &   &   &   &  \phantom{-}1 &   &   &   &   &   &  \\
      &   &   &   & 2 &   &   &   &   &   &   &   &   &  \\
      &   &   &   &   & \ddots &   &   &   &   &   &   &   &  \\
      &   &   &   &   &   & 2 &   &   &   &   &   &   &  \\
     &   &   &   &   &   &  &\coolunder{t_0}{}  &   &   &   &   &   &  \\
\end{array}
\right)
$
\end{adjustbox}

\vspace{1cm}
We began with a smooth map $\nabla f:W\to T^\vee W$, and now we can check to see whether the image of this map is transverse to the inclusion of the zero section of $T^\vee W$, $i_0:W\to T^\vee W$. 
If $\ell$ is in $\nabla f(W) \cap i_0 W$, then $\nabla f(W)$ is transverse to  $i_0 W$ if, for all $a,b,\ell$ so that $\nabla f(a)=i_0(b)=\ell$,
$$\text{Im}(D(\nabla f)(a))\oplus  \text{Im}(D(i_0)(b)) \twoheadrightarrow T_p(T^\vee W).$$
It is clear that the intersection $\ell\in \nabla f(W) \cap i_0 W$ is the set of critical points of $f$, and as these points are either Morse critical points, or birth-death singularities, we handle each of these cases independently. 
In the event that $\ell$ is a Morse critical point, it is a standard exercise that the map above is surjective. 
The case of a birth-death singularity is identical, except in the row labelled by $\frac{\partial}{\partial x_{i+1}}$. 
At the birth-death singularity, the entry $6x_{i+1}$ vanishes, and if it were not for the 1 in the $t_0$ entry of the row, there would be no image in the 1-dimensional subspace spanned by $\frac{\partial}{\partial x_{i+1}}$ of the matrix above. 
So we do have surjectivity and thus transversality, but only because of the derivative in the unfolding direction $t_0$. 
So transversality is a direct consequence of the unfolding behavior of a parametrized family of generalized Morse functions. 
We summarize this discussion in the lemma and corollary below:

\begin{lemma}
For $f:(W,\partial_0W)\to (I,0)$ a fiberwise generalized Morse function on a  smooth fiber bundle $W\to B$, the section $\nabla f:W\to T^\vee W$ is transverse to the zero section of the vertical tangent bundle of $W$.
\end{lemma}

\begin{corollary}
If $\Sigma_f$ denotes the critical locus of a fiberwise generalized Morse function $f$, then the normal bundle $\nu(\Sigma_f)\to\Sigma_f$ to the embedding $\Sigma_f\hookrightarrow W$ is isomorphic to the restriction of the vertical tangent bundle of $W$ to $\Sigma_f$, $T^\vee W|_{\Sigma_f}$.
\end{corollary}

\subsection{Cancelling critical points and ghost sets}
\label{subsec:ghosts}

In this section we will consider the submanifold $\Sigma_f$ of $W$, introduce notation for the submanifolds made up of Morse critical points and birth-death critical points, and discuss how to perturb the critical locus to facilitate the proofs of the fiberwise Poincar\'e–Hopf theorems appearing later in this paper.

Let $Z_i$ denote the submanifold of Morse critical points of degree $i$, where the degree is the number of negative eigenvalues of the Hessian of $f$ at any point in $Z_i$. In general the collection of such critical points may have more than one component, but we will not introduce extra notation for this level of generality. Instead, we assume that $Z_i$ is connected, and note that all proofs in this paper can easily be generalized to accomodate multiple components. The collection of all such $Z_i$ is denoted $\sS(\Sigma_f)$.

The submanifolds $Z_i$ and $Z_{i+1}$ share a common boundary which we denote by $Z^1_i$. The submanifold $Z^1_i$ contains the birth-death critical points of degree $i$. Again, there may be more than one component of birth-death critical points of degree $i$, but we elect not to work at that level of generality.

\medskip

The diagram below depicts how $Z_i$, $Z_{i+1}$, and $Z_i^1$ are arranged in the neighborhood of a birth-death critical point. 

\begin{center}
\begin{figure}[h]
\begin{overpic}[abs,width=6cm]{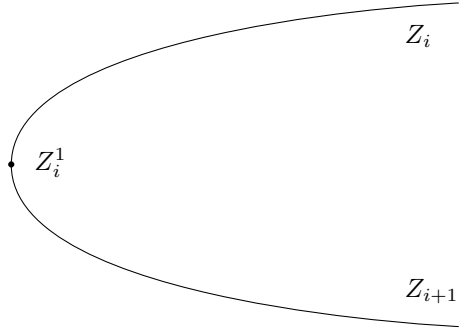}
\put(150,108){$Z_i$}
\put(150,12){$Z_{i+1}$}
\put(10,60){$Z_i^1$}
\end{overpic}
\caption{The critical submanifolds $Z_i$ and $Z_{i+1}$ share a common boundary $Z_i^1$.}
\end{figure}
\end{center}

The image of the birth-death set in $B$, $p(Z_1^i)$, is known as the \textit{bifurcation set}. 
The bifurcation set is a codimension one submanifold of $B$. 
In the neighborhood of a birth-death singularity, the points locally given by $x_i=0$ and $-\epsilon<t_0<0$ are inflection points on which the second derivative of $f$ in the vertical direction vanishes~\cite{Igu05}.
We call these points \textit{ghost points}, and they allow us to define a \textit{ghost set}, $Z_i^g$, which is formally a lift of a one-sided collar neighborhood of the bifurcation set. 
The ghost set is transverse to $Z$, as in the following figure.

\begin{center}
\begin{figure}[H]
\begin{overpic}[abs,width=8cm]{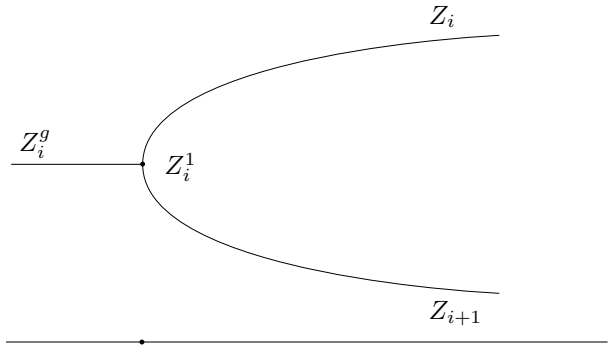}
\put(5,74){$Z_i^g$}
\put(60,65){$Z_i^1$}
\put(160,122){$Z_i$}
\put(160,10){$Z_{i+1}$}
\end{overpic}
\caption{The ghost set is transversally attached to the critical locus at the birth-death set.}
\end{figure}
\end{center}

The ghost set is used to locally perturb $Z$ so that the critical points of the generalized Morse function do not cancel over the ghost set. 
In particular, we consider manifolds with corners $\widehat{Z_i}:=Z_i\cup Z_i^g$ and $\widehat{Z_{i+1}}:=Z_{i+1}\cup Z_i^g$ and we smooth both of these manifolds to obtain manifolds $\widetilde{Z_i}$ and $\widetilde{Z_{i+1}}$ that are locally diffeomorphic to $B$. These are depicted below. 

\begin{center}
\begin{figure}[H]
\begin{overpic}[abs,width=8cm]{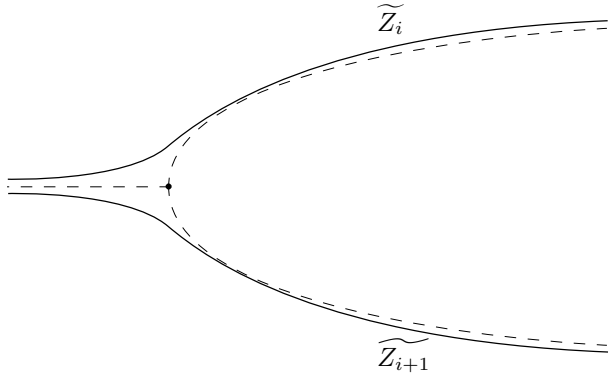}
\put(140,122){$\widetilde{Z_i}$}
\put(140,-5){$\widetilde{Z_{i+1}}$}
\end{overpic}
\caption{The manifolds with corners $\widehat{Z_i}$ and $\widehat{Z_{i+1}}$ are smoothed out to $\widetilde{Z_i}$ and $\widetilde{Z_{i+1}}$.}
\end{figure}
\end{center}

The outcome of this perturbation is that any sufficiently small simplex in the base which intersects the bifurcation set has the same number of critical points over each point in the simplex. 
This essential property of ghosts is used in~\cite{Igu05} in the proof of the `transfer theorem', and in~\cite{Ohr19} to give a combinatorial description of the Becker–Gottlieb transfer.

\subsection{Stratified deformations of critical loci}
\label{subsec:strat_def}

In this section we define \textit{stratified subsets} and \textit{stratified deformations}, and we construct a particular stratified deformation that will be used in Section \ref{subsec:PHSD} to prove Theorems \ref{thm:PHHSD} and \ref{thm:PHcubeSD}. 

\begin{defn}
A stratified subset of a smooth bundle $p:W\to B^k$ is a pair $(\Sigma, \psi)$ where $\Sigma$ is a compact smooth $k$-dimensional manifold together with a map $\rho:\Sigma\to B$ and a tangential structure $\psi: \Sigma\to X$. 
The map $\rho$ is everywhere smooth, but may admit birth-death singularities locally given by $\rho(x_1,\cdots, x_k) = (x_1^2,x_2,\cdots, x_k)$. 
These singularities form a $k-1$ dimensional submanifold of $\Sigma$. 
\end{defn}

We give two examples of stratified subsets, the first explains how to obtain the canonical example of a stratified subset from a fiberwise generalized Morse function. The second example introduces a special type of stratified subset called an \textit{immersed lens}. The particular stratified deformation discussed below begins with a critical locus of a fiberwise generalized Morse function and deforms it into a disjoint union of immersed lenses.

\begin{example}\label{ex:GMF_SD}
For this example the pair $(\Sigma,\psi)$ will be the stratified subset corresponding to the critical locus of a fiberwise generalized Morse function $f:(W,\partial_0W)\to(I,0)$ on the bundle $p:W\to B$. In this case, $\Sigma$ is the critical locus of $f$, and the map $\rho$ is the projection $p$ restricted to $\Sigma$.
The map $\psi$ will be a map $\Sigma\to BO\times BO$ classifying the stable negative eigenspace bundle of $f$ in the first component, and the stable positive eigenspace bundle of $f$ in the second component.
\end{example}

\begin{rmk}
There are two different stratifications on a stratified subset. The titular stratification refers to the stratification by dimension: each stratum is either of dimension $k$ or dimension $k-1$. Often this will not be the stratification that we are interested in. Instead, we will make use of the degree stratification which distinguishes by the degree of their critical points. For instance, the submanifold of the critical locus containing those critical points of degree $i$, previously denoted $Z_i$, is a stratum of the degree-wise stratification of $\Sigma_f$. We will denote the set of such strata by $\sS(\Sigma_f)$. 
\end{rmk}

\noindent The following is an example of a stratified subset concentrated in two degrees. 

\begin{example}[Immersed Lenses - p.70 in \cite{Igu05}]
Let $V$ be a compact connected $k$ manifold with boundary so that $V$ is immersed in $B^k$. 
Let $\psi_1, \psi_2:V\to X$ be continuous maps which are trivial on $\partial V$. 
Then the \emph{immersed lens} $L_i(V,\psi_1,\psi_2)$ is defined to be the stratified subset $(L,\psi_L)$ where $L$ is the double of $V$ in indices $i$ and $i+1$, and $\psi_L$ is $\psi_1$ on the lower stratum and $\psi_2$ on the upper stratum.  
\end{example}

\begin{defn}[p. 67 in~\cite{Igu05}]
A stratified deformation between stratified subsets $(\Sigma,\psi)$ and $(\Sigma',\psi')$ of $p:W\to B$ with coefficients in $X$ is a stratified subset $(S,\Psi)$ of $p\times I:W\times I\to B\times I$ with coefficients in $X$ so that the restrictions of $(S,\Psi)$ to  $p\times 0:W\times 0\to B\times 0$ and $p\times 1: W\times 1\to B\times 1$ are $(\Sigma,\psi)$ and $(\Sigma',\psi')$. 
In the event that $(\Sigma,\psi)$ and $(\Sigma',\psi')$ are related by a stratified deformation we say that they belong to the same \textit{ stratified deformation class} and use the notation  $(\Sigma,\psi)\sim(\Sigma',\psi')$.
\end{defn}

\begin{rmk}
Note that when considering a stratified deformation of a critical locus of a fiberwise generalized Morse function, the end result of the deformation may not necessarily be realized as the critical locus of a fiberwise generalized Morse function. 
When referencing the collection of strata of the stratified subset $(\Sigma,\psi)$ distinguished by their degree (in this case the dimension of the negative eigenspace bundle), we use the notation $\sS(\Sigma)$. When $\Sigma$ is the critical locus $\Sigma_f$, $\sS(\Sigma)$ is identical to $\sS(\Sigma_f)$. 
\end{rmk}

\noindent For the remainder of this section we fix $(\Sigma,\psi)$ as in Example~\ref{ex:GMF_SD}.

\begin{lemma}\label{lem:SD1}
There exists a stratified deformation between stratified subsets $(\Sigma,\psi)$ and $(\Sigma',\psi')$, so that the degree-wise strata of $(\Sigma',\psi')$ are concentrated in two consecutive degrees. 
Furthermore, each component of the lower stratum of $\Sigma'$ lies in a contractible subset of $\Sigma'$.
\end{lemma}

\begin{rmk}
Lemma~\ref{lem:SD1} is an excerpt from the proof of the transfer theorem in~\cite{Igu05}. 
In particular, the statement is identical to Step (c) on p.70, the proof of which appears on pages 71-73. 

\medskip

Briefly, the strategy of the proof is to first add and delete twisted lenses, immersed lenses for which $\psi_1$ is the same as $\psi_2$ after composition with the fold map, to concentrate the stratified subset into two consecutives degrees. 
The stratified deformations obtained by adding and deleting the twisted lenses reduce the number of components in the top degree by one. 
An inductive argument starting in the minimal stratum will then concentrate all strata in two degrees. 

\medskip

The next task is to prove that the lower stratum lies in a contractible subset of $\Sigma'$. 
To do this, we choose a triangulation of $\Sigma'$, and do a deformation on each simplex. On the zero simplices the idea is to add a lens above a zero simplex, give a stratified deformation that cancels the lower stratum of this lens to obtain a `mushroom', and then observe that the mushroom has the desired property: the `$-$' stratum on top of the mushroom (as well as it's boundary) lies in a contractible subset. 
We give a pictorial version of this stratified deformation in the figure below. 
This construction, as well as the inductive constructions for higher simplices, also appears with pictures in the proof of Lemma 3.2.1 in~\cite{GI14}. 

\begin{figure}[h]
\begin{center}
\begin{overpic}[width=14cm]{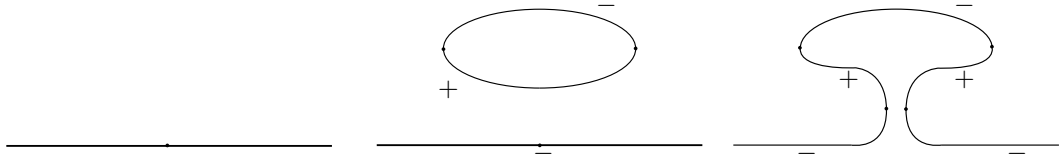}
\put(50,-1){$-$}
\put(41,5){$+$}
\put(56,13){$-$}
\put(75,-1){$-$}
\put(95,-1){$-$}
\put(79,6){$+$}
\put(90,6){$+$}
\put(90,13){$-$}
\end{overpic}
\caption{The stratified deformation introduces a lens above a designated point in $\Sigma_-$ and then cancels the + and - strata to obtain a `mushroom'.}
\end{center}
\end{figure}

\end{rmk}

\begin{lemma}[Lemma 5.7 in~\cite{Igu05}]\label{lem:SD2}
The stratified subset $(\Sigma',\psi')$ can be deformed into a stratified subset $(\Sigma_{\text{SD}},\psi_{\text{SD}})$ presented as a disjoint union of immersed lenses and components on which $\psi_{\text{SD}}$ is trivial. 
Furthermore, $\psi_{\text{SD}}$ is trivial on the lower stratum of each of the immersed lenses. 
\end{lemma}

\noindent Let $\Lambda$ denote the component of $\Sigma_\sd$ on which $\psi_\sd$ is trivial. 

\begin{lemma}\label{lem:null_deformable}
An integer multiple of the stratified subset $(\Lambda,*)$ is stratified null-deformable. 
\end{lemma}
\begin{proof}
Since $\Lambda$ is concentrated in two degrees, $\Lambda$ gives a map from $B$ into the configuration space of positive and negative particles which is homotopy equivalent to $QS^0$. Since $QS^0$ is rationally trivial in degrees greater than 0, and this map lands in the zero component of $\pi_0(QS^0)=\Z$, some positive integer multiple of $(\Lambda,*)$ must be stratified null-deformable.  
\end{proof}

\begin{rmk}
In Section \ref{subsec:PHSD} the stratified deformation between $(\Sigma,\psi)$ and $(\Sigma_{\text{SD}}, \psi_{\text{SD}})$ is used to prove a fiberwise Poincar\'e–Hopf theorem that factors the Becker–Gottlieb transfer in terms of $(\Sigma_{\text{SD}}, \psi_{\text{SD}})$. In the proof of Theorem \ref{thm:PHHSD} we make use of ghost sets on $\Sigma_{\text{SD}}$, which are defined on arbitrary stratified subsets identically to how they are defined on $\Sigma_f$. 
\end{rmk}

\section{Fiberwise Poincar\'e–Hopf theorems}
\label{sec:PHthms}

This section contains several fiberwise Poincar\'e–Hopf theorems, the main technical results of this paper. Roughly speaking, a fiberwise Poincar\'e–Hopf theorem computes a characteristic in terms of Morse theoretic data, in analogy with the classical Poincar\'e–Hopf theorem. The results in this section are preceded in the literature by computations of the Becker–Gottlieb transfer due to~\cite{BM76} and~\cite{Dou06}. 

\medskip

In Subsection~\ref{subsec:PHthms_setup} we fix notation and introduce the indexing categories that are used to give refinements of our characteristics as Euler sections. These are the Poincar\'e duals of the usual characteristics. The constructions of the characteristics appearing in this subsection are used several times over in the subsequent subsections. 

\medskip

In Subsection~\ref{subsec:QthyPH} we prove a fiberwise Poincar\'e–Hopf theorem for the Becker–Gottlieb transfer. To be exact, we factor the Poincar\'e dual of the Becker–Gottlieb transfer, and Euler section $e^d(p)$, in terms of the critical locus of a fiberwise generalized Morse function.  

\medskip

In Subsection~\ref{subsec:AthyPH} we prove a fiberwise Poincar\'e–Hopf theorem for the excisive A-theory Euler characteristic. To be exact once again, we factor the Poincar\'e dual of the excisive A-theory Euler characteristic, and Euler section $e^t(p)$, in terms of the critical locus of a fiberwise generalized Morse function. We also prove that the results of Subsections~\ref{subsec:QthyPH} and \ref{subsec:AthyPH} are compatible. 

\medskip

In Subsection~\ref{subsec:PHSD} we generalize the results of the previous two sections to an arbitrary stratified deformation of the critical locus of a fiberwise generalized Morse function. We also translate the theorems of this section into rational formulas in $\pi_0$ for use in Subsection~\ref{subsec:PHTheta} and Section~\ref{sec:ThetaCalc}.

\medskip

In Subsection~\ref{subsec:PHTheta} we use the rational formulas of the previous section, as well as the stratified deformation constructed in Subsection~\ref{subsec:strat_def} to prove a fiberwise Poincar\'e–Hopf theorem for the smooth structure class. In contrast to the previous fiberwise Poincar\'e–Hopf theorems, Theorem~\ref{thm:PHTheta} is a rational statement. 

The contents of this section build towards Theorem~\ref{thm:PHTheta}, which is the only theorem from this section used in Section~\ref{sec:ThetaCalc}.

\subsection{Definitions and a recollection of Dwyer–Weiss–Williams index theory}
\label{subsec:PHthms_setup}

For the subsections that follow we will fix a smooth manifold bundle $p:W\to B$, where $W$ is a compact smooth manifold of dimension $m$, $B$ is a compact smooth manifold of dimension $k$, and the fiber of $p$ is a compact smooth manifold $F$ of dimension $n$. 
We fix an embedding of $W$ into $B\times\R^d$ over $B$ for $d$ large. We also fix a fiberwise generalized Morse function $f:W\to \R$, with critical locus $\Sigma^k_f$ and vertical gradient vector field $X:=\nabla^\vee f$.
We denote by $\pi$ the restriction of the projection $p$ to $\Sigma_f$. 
We denote by $\eta$ the normal bundle to the embedding of $\Sigma_f$ into $W$, which is an $(m-k)$-plane bundle on $\Sigma_f$.
The tubular neighborhood of $\Sigma_f$ in $W$ is a disk bundle $q:D\Sigma_f\to \Sigma_f$, and the restriction of $p$ to $D\Sigma_f$ is a map $\psi:D\Sigma_f\to B$. 
These choices are summarized in the following diagram. 
We also fix the notation $\tau$ and $\nu$ for, respectively, the vertical tangent bundle and the vertical normal bundle of $W$.

$$
\xymatrix{
&&&&F^n\ar[d]\ar[r]&\R^d\ar[d]\\
\Sigma_f^k\ar@/_/[drrrr]_\pi\ar@<.5ex>[rr]&&D\Sigma_f^m\ar[rrd]^\psi\ar@<.5ex>[ll]^q\ar[rr]&&W^m\ar[d]_p\ar[r]&B\times\R^d\ar[ld]\\
&&&&B^k\\
}
$$

\medskip

Next we introduce indexing categories associated to the manifolds above, which will be used to construct the characteristics considered in subsequent subsections. 

\begin{defn}\label{def:disksInBase}
The category $\disk_k^{B/}$ is the category of one point compactifications of $k$-disks embedded in $B$. More precisely, an object $U\in\disk_k^{B/}$ is the one point compactification of an open disk $\R^k$ embedded in $B$. An open embedding $\R^k\hookrightarrow \R^k\hookrightarrow B$ gives rise to a morphism $U'\to U$ given by the one point compactification of the open embedding $\R^k\hookrightarrow\R^k$.
\end{defn}

\begin{rmk}
To ease the notation used for indexing the diagrams the follow, we will refer to objects of $\disk_k^{B/}$ and related categories without the bullet superscript, e.g. $U$ instead of $U^\bullet$. However, when defining morphisms we will be precise. In particular, we will use $U$ when referring to the open $k$-plane embedded in $B$, and use $U^\bullet$ when referring to the one point compactification of $U$. 
\end{rmk}

\begin{rmk}
The category $\disk_k^{B/}$ of Definition~\ref{def:disksInBase} is similar to the category $\left(\Disk_n^+\right)^{B_*/}$ used in~\cite{AF19}, with the only difference being that the $n$-disks used here have only one component. 
\end{rmk}

\begin{defn}\label{defn:disksInFiber}
For fixed $U\in\disk_k^{B/}$, we define the category $\disk_m^{p^{-1}U/}$ to be the category of one point compactifications of $m$-disks in $p^{-1}U$. More precisely, an object $V$ in $\disk_m^{p^{-1}U/}$ is the one point compactification of an $m$ disk $\R^k\times\R^n\cong \R^m$ embedded in $p^{-1}U$ so that the composition $\R^k\times\R^n\to p^{-1}U\xrightarrow{p}U$ factors through the projection $\R^k\times\R^n\to \R^k$. A morphism from $V'$ to $V$ is given by the one point compactification of an open embedding $\R^k\times\R^n\to\R^k\times \R^n$ that commutes with the projection maps to $\R^k$. 
\end{defn}

\begin{defn}
The category $\disk_m^{\psi^{-1}U/}$ is a subcategory of $\disk_m^{p^{-1}U/}$ containing those disks $V$ for which the embedding $\R^k\times\R^n\hookrightarrow{p^{-1}U}\hookrightarrow W$ factors through the embedding $D\Sigma_f\hookrightarrow W$.
\end{defn}

The proposition below indicates how the categories $\disk_k^{B/}$ and $\disk_m^{p^{-1}U/}$ are used to model the homotopy types of $\Gamma_BQ_B(W)$, $\Gamma_BA^\%_B(W)$, and $\Gamma_B\h^\%_B(W)$. These are the spaces of sections of the fiberwise homology bundles whose fibers are $Q(F_+):=\Omega^\infty(F_+\wedge\S)$ for $\S$ the sphere spectrum, $A^\%(F):= \Omega^\infty(F_+\wedge A(*))$ for $A(*)$ the algebraic K-theory of spaces functor evaluated at a point, and $\h^\%(F):=\Omega^\infty(F_+\wedge \h(*))$ for $\h(*)$ the stable h-cobordism space of a point. 

\begin{prop}
The following are homotopy equivalences: 
\begin{align}
\Gamma_BQ_B(W)&\xrightarrow{\simeq}\holim_{U\in \disk_q^{B/}}\holim_{V\in\disk_m^{p^{-1}U/}}
\Omega^\infty(V^\bullet\wedge \S)\\
\Gamma_BA^\%_B(W)&\xrightarrow{\simeq}\holim_{U\in \disk_q^{B/}}\holim_{V\in\disk_m^{p^{-1}U/}} \Omega^\infty(V^\bullet\wedge A(*))\\
\Gamma_B\h^\%_B(W)&\xrightarrow{\simeq}\holim_{U\in \disk_q^{B/}}\holim_{V\in\disk_m^{p^{-1}U/}} \Omega^\infty(V^\bullet\wedge \h(*))
\end{align}
\end{prop}

\begin{proof}
These homotopy equivalences follow from Poincar\'e duality, e.g. \cite{DWW03} or nonabelian Poincar\'e duality from \cite{AF19}. 
\end{proof}

\subsubsection{Constructions of characteristics}

We will now construct refinements of the Becker–Gottlieb transfer $\tr(p)\in\Gamma_BQ_B(W)$ and the excisive A-theory Euler characteristic $\chi^\%(p)\in\Gamma_BA_B^\%(W)$. First we recall the Euler sections of~\cite{Bec70} and~\cite{DWW03}. 

\medskip

Let $\gamma(n)$ be the tautological bundle on $\btop(n)$, and let $[e^t_n]\in H^{\gamma(n)}(\btop(n);A(*))$, the cohomology of $\btop(n)$ with twisted coefficients in the spectrum $A(*)$, denote the generalized Becker–Euler class defined in~\cite{Bec70}. The class $[e^t_n]$ is refined in~\cite{DWW03} to a section $e_n^t$ of the bundle with base $\btop(n)$ and fiber $\Omega^\infty(\gamma(n)_x^\bullet\wedge A(*))$ over $x\in \btop(n)$. Given a manifold $M$ with tangent Euclidean bundle $\tau$ classified by a map $M\to \btop(n)$, the associated Euler section $e^t_n(\tau)$ is defined as the pullback of $e^t_n$ to a section of the bundle over $M$ with fibers $\Omega^\infty(\tau^\bullet_x\wedge A(*))$ over $x\in M$. 

\medskip

Similarly, let $\epsilon(n)$ be the tautological bundle on $\bo(n)$, and consider the Becker–Euler class $[e^d_n]\in H^{\epsilon(n)}(\bo(n);\S)$, the cohomology of $\bo(n)$ with twisted coefficients in the sphere spectrum. The class $[e^d_n]$ is refined in~\cite{DWW03} to a section $e_n^d$ of the bundle with base $\bo(n)$ and fiber $\Omega^\infty(\epsilon(n)_x^\bullet\wedge \S)$ over $x\in \bo(n)$. Given a manifold $M$ with tangent bundle $\tau$ classified by a map $M\to \bo(n)$, the associated Euler section $e^d_n(\tau)$ is defined as the pullback of $e^d_n$ to a section of the bundle over $M$ with fibers $\Omega^\infty(\tau^\bullet_x\wedge \S)$ over $x\in M$. 

\medskip

For $U\in\disk_k^{B/}$ and $V\in\disk_m^{p^{-1}U/}$, let $c$ denote the Thom collapse map $U^\bullet\wedge S^d\to V^\bullet\wedge\Th(\nu|_V)$ associated to the embedding $V\hookrightarrow U\times \R^d$. Let $e^d_n(\tau):\Th(\nu|_V)\wedge S^0\xrightarrow{\id\wedge e_n^d(\tau)} \Th(\nu|_V)\wedge\Th(\tau|_V)\wedge\S$ denote the restriction of the Euler section $e^d_n(\tau)$ on $M$ to $V$. At a point in $\Th(\nu|_V)$ over $x\in V$, the map $e_n^d(\tau)$ is the Euler section at $x$, i.e. a map $S^0\to \tau^\bullet_x\wedge\S$.  Then we consider the composition below:

\begin{equation}\label{eqn:char_map_d}
    U^\bullet\wedge S^d\wedge S^0\xrightarrow{c\wedge \id}V^\bullet\wedge\Th(\nu|_V)\wedge S^0\xrightarrow{\id\wedge e_n^d(\tau)}V^\bullet\wedge\Th(\nu|_V)\wedge\Th(\tau|_V)\wedge \S
\end{equation}

\medskip\noindent Identifying $\Th(\nu|_V)\wedge\Th(\tau|_V)$ as $\S^d$ and taking adjoints, we equivalently have a map

$$
U^\bullet\wedge S^0\to \Omega^d\Omega^\infty(V^\bullet\wedge S^d\wedge\S).
$$

\medskip\noindent For simplicity, we further compose with a homotopy equivalence given by the inclusion $\Omega^d\Omega^\infty(V^\bullet\wedge S^d\wedge\S)\to \Omega^\infty(V^\bullet\wedge\S)$ to obtain a map 

$$
U^\bullet\wedge S^0\to \Omega^\infty(V^\bullet\wedge\S)
$$

\medskip\noindent The construction above is natural in $V$, and thus yields a map

$$
U^\bullet\wedge S^0\to \displaystyle\holim_{V\in\disk_m^{p^{-1}U/}}\Omega^\infty(V^\bullet\wedge\S)
$$

\medskip\noindent Naturality in $U$ then produces a map

$$
S^0\xrightarrow{e^d(p)} \displaystyle\holim_{U\in\disk_k^{B/}}\displaystyle\holim_{V\in\disk_m^{p^{-1}U/}}\Omega^\infty(V^\bullet\wedge\S)
$$

\medskip\noindent which we denote by $e^d(p)$ and refer to as the \textit{Euler section of $p$ in $\Gamma_BQ_B(W)$}.

\medskip

\medskip

\noindent Similarly, considering the composition 
\begin{equation}\label{eqn:char_map_t}
    U^\bullet\wedge S^d\wedge S^0\xrightarrow{c\wedge \id}V^\bullet\wedge\Th(\nu|_V)\wedge S^0\xrightarrow{\id\wedge e_n^t(\tau)}V^\bullet\wedge\Th(\nu|_V)\wedge\Th(\tau|_V)\wedge A(*)
\end{equation}
results in a map 
$$
S^0\xrightarrow{e^t(p)} \displaystyle\holim_{U\in\disk_k^{B/}}\displaystyle\holim_{V\in\disk_m^{p^{-1}U/}}\Omega^\infty(V^\bullet\wedge A(*))
$$

\medskip\noindent which we denote by $e^t(p)$ and refer to as the \textit{Euler section of $p$ in $\Gamma_BA^\%_B(W)$}.

\subsubsection{Recollections from Dwyer–Weiss–Williams}

\begin{prop}\label{prop:DWW_vanishing}
With the vertical map induced by the unit map $\eta:\S\to A(*)$, the following diagram is homotopy commutative:
\medskip
$$
\xymatrix{
S^0\ar[rrrr]^{e^d(p)}\ar[rrrrd]^{e^t(p)}&&&&\displaystyle\holim_{U\in \disk_k^{B/}}\displaystyle\holim_{V\in\disk_m^{p^{-1}U/}}
\Omega^\infty(V^\bullet\wedge \S)\ar[d]^{\eta_*}\\
&&&&\displaystyle\holim_{U\in \disk_k^{B/}}\displaystyle\holim_{V\in\disk_m^{p^{-1}U/}} \Omega^\infty(V^\bullet\wedge A(*))
}
$$
\end{prop}

\begin{proof}
It suffices to construct a path relating the universal Euler sections $\eta_*e_n^d$ and $e_n^t$. This is imprecisely Theorem 4.10 and precisely Theorem 4.13 in~\cite{DWW03}.
\end{proof}

\noindent The fiberwise Poincar\'e duality map
$$
\p:\holim_{V\in\disk_m^{p^{-1}U/}}\Omega^\infty(V^\bullet\wedge \J)\xrightarrow{\simeq}\Omega^\infty((p^{-1}U)^\bullet\wedge \J)
$$
is a homotopy equivalence for any spectrum $\J$ (see, e.g.~\cite{AF19} or~\cite{DWW03}).

\begin{prop}\label{prop:DWW_index_thms}
There is a canonical path between $\p e^d(p)$ and $\tr(p)$ in $$
\displaystyle\holim_{U\in\disk_k^{B/}}\Omega^\infty((p^{-1}(U))^\bullet\wedge \S)\simeq\Gamma_BQ_B(W)
$$
There is also a canonical path between 
$\p e^t(p)$ and $\chi^\%(p)$ in $$
\displaystyle\holim_{U\in\disk_k^{B/}}\Omega^\infty((p^{-1}(U))^\bullet\wedge A(*))\simeq\Gamma_BA^\%_B(W)
$$
\end{prop}
\begin{proof}
The first sentence is Theorem 5.4 in~\cite{DWW03}. The second sentence is Theorem 3.18 in~\cite{DWW03}.
\end{proof}

\begin{prop}
There is a canonical path betwween $\eta\tr(p)$ and $\chi^\%(p)$ in $\Gamma_BA^\%_B(W)$. In particular, the following diagram is homotopy commutative.
$$
\xymatrix{
S^0\ar[rr]^{\tr(p)}\ar[rrd]_{\chi^\%(p)}&&\Gamma_BQ_B(W)\ar[d]^\eta\\
&&\Gamma_BA^\%_B(W)
}
$$
\end{prop}
\begin{proof}
This follows from combining Propositions~\ref{prop:DWW_vanishing} and~\ref{prop:DWW_index_thms}.
\end{proof}

\subsection{Fiberwise Poincar\'e--Hopf Theorem for the Becker--Gottlieb transfer}
\label{subsec:QthyPH}

In this section we factor the Becker–Gottlieb transfer in terms of the critical locus $\Sigma_f$ of the fiberwise generalized Morse function $f:W\to [0,1]$ and the vertical gradient vector field $X:=\gamma^\vee f$. 

\medskip\noindent Recall the notation $Z_i$ for the connected submanifold of $\Sigma_f$ containing those critical points of degree $i$. The submanifolds $Z_i$ and $Z_{i+1}$ share a boundary $Z_i^1$ consisting of birth-death critical points. The set $\sS(\Sigma_f)$ is defined to be the set of all such $Z_i$. In this section, we will make use of manifolds $\tZ_i$ obtained by perturbing $Z_i$ over the ghost set as in Section \ref{subsec:ghosts}. Note that we will often omit the degree $i$ subscript from $\tZ_i$ when it is not essential. We denote by $\pi_\tZ$ the local diffeomorphism given by the restriction of $p$ to $\tZ$.

\medskip\noindent We begin by introducing two maps associated to $\tZ$. The first is an Euler section associated to $\tZ$, and the second is an index map associated to $\tZ$. After these are defined, we prove the main theorem of this section. 

\subsubsection{The Euler section associated to $\tZ$}

We begin by giving definitions of the categories used to approximate $\tZ$.

\begin{defn}
The category $\disk_k^{\tZ/}$ is the category of one point compactifications of $k$ disks embedded in $\tZ$. More precisely, an object $V\in\disk_k^{\tZ/}$ is the one point compactification of an open disk $\R^k$ embedded in $\tZ$. An open embedding $\R^k\hookrightarrow\R^k\hookrightarrow \tZ$ gives rise to a morphism $V'\to V$ given by the one point compactification of the open embedding $\R^k\hookrightarrow\R^k$. 
\end{defn}

\begin{defn}\label{defn:diskZ_U}
For $U\in\disk_k^{B/}$, the category $\disk_k^{\tZU}$ is the subcategory of $\disk_k^{\tZ/}$ consisting of those objects $V$ for which the associated $\R^k$ embedded in $\tZ$ maps into $U$ under the projection map $p$, which restricts to a local homeomorphism on $\R^k$. 
\end{defn}

Next we define a characteristic $e^d(\pi_{\widetilde{z}})$ associated to $\widetilde{Z}$ in much the same way as the characteristic $e^d(p)$ was defined. Consider the composition  

\begin{equation}\label{eqn:char_map_d_Z}
    U^\bullet\wedge S^d\wedge S^0\xrightarrow{c\wedge \id}V^\bullet\wedge\Th(\eta\oplus\nu|_V)\wedge S^0\xrightarrow{\id\wedge e_n^d(\tau)}V^\bullet\wedge\Th(\eta\oplus\nu|_V)\wedge\Th(\tau|_V)\wedge \S
\end{equation}

\medskip\noindent As in the previous section, the composition~(\ref{eqn:char_map_d_Z}) is used to construct a map

$$
S^0\xrightarrow{e^d(\pi_\tZ)} \displaystyle\holim_{U\in\disk_k^{B/}}\displaystyle\holim_{V\in\disk_k^{\tZU}}\Omega^\infty(V^\bullet\wedge\S)
$$

\noindent Aggregate over all $Z\in\sS(\Sigma_f)$, we have a map

$$
S^0\xrightarrow{\prod_{Z\in\sS(\Sigma_f)}e^d(\pi_\tZ)} \displaystyle\prod_{Z\in\sS(\Sigma_f)}\displaystyle\holim_{U\in\disk_k^{B/}}\displaystyle\holim_{V\in\disk_k^{\tZU}}\Omega^\infty(V^\bullet\wedge\S)
$$

\subsubsection{The index map associated to $\tZ$}

At a point $z\in Z$, we can consider the derivative of the gradient vector field $X$ at $z$, a map $dX_z:\tau_z\to\tau_z$. This induces a self-map on the one point compactification of $\tau_z$, and thus a map $dX:\Th(\tau_Z)\to\Th(\tau_Z)$. For any $V\in\disk_k^{\tZ/}$, we can restrict the map $dX$ to $V$ to obtain a map $dX|_V:\Th(\tau|_V)\to\Th(\tau|_V)$. Then the local map 

\begin{equation}\label{eqn:index_local_d}
V^\bullet\wedge\Th(\eta\oplus\nu|_V)\wedge\Th(\tau|_V)\wedge \S\xrightarrow{\id\wedge\id\wedge dX|_V\wedge\id}
V^\bullet\wedge\Th(\eta\oplus\nu|_V)\wedge\Th(\tau|_V)\wedge \S
\end{equation}

\noindent induces a map

\begin{equation}\label{eqn:d_index_map}
\displaystyle\holim_{U\in\disk_k^{B/}}\displaystyle\holim_{V\in\disk_k^{\tZU}}\Omega^\infty(V^\bullet\wedge\S)\xrightarrow{\ind^d_\tZ}
\displaystyle\holim_{U\in\disk_k^{B/}}\displaystyle\holim_{V\in\disk_k^{\tZU}}\Omega^\infty(V^\bullet\wedge\S)
\end{equation}

\noindent which we denote by $\ind_\tZ^d$ and refer to as the index map on $\tZ$ with coefficients in $\S$.

\medskip

\noindent Aggregate over all $Z\in \sS(\Sigma_f)$ we have a map

\begin{equation}\label{eqn:d_index_map}
\displaystyle\prod_{Z\in\sS(\Sigma_f)}\displaystyle\holim_{U\in\disk_k^{B/}}\displaystyle\holim_{V\in\disk_k^{\tZU}}\Omega^\infty(V^\bullet\wedge\S)\xrightarrow{\prod_{Z\in\sS(\Sigma_f)}\ind^d_\tZ}
\displaystyle\prod_{Z\in\sS(\Sigma_f)}\displaystyle\holim_{U\in\disk_k^{B/}}\displaystyle\holim_{V\in\disk_k^{\tZU}}\Omega^\infty(V^\bullet\wedge\S)
\end{equation}

\noindent The following lemma will be used in Section~\ref{subsec:PHSD}.

\begin{lemma}\label{lem:ind_on_htpy_gps}
For $Z\in\sS(\Sigma_f)$ of degree $j$, $\ind_\tZ^d$ is multiplication by $(-1)^j$ on homotopy groups. 
\end{lemma}
\begin{proof}
This is the correction in \cite{MP89} to Theorems 2.11 and 3.11 in \cite{BM76} (see end of Section 2 in~\cite{MP89}) . 
\end{proof}

\subsubsection{Factoring the Becker–Gottlieb transfer}

\begin{theorem}\label{thm:PHQthy}
The diagram below is homotopy commutative.

\begin{adjustbox}{width=\columnwidth,center}
\xymatrix{
S^0\ar[dddd]_{\prod_{Z\in\sS(\Sigma_f)}e^d(\pi_\tZ)}\ar[rrrr]^{e^d(p)}&&&&\displaystyle\holim_{U\in \disk_k^{B/}}\displaystyle\holim_{V\in\disk_m^{p^{-1}U/}}
\Omega^\infty(V^\bullet\wedge \S)\\
\\
\\
\\
\displaystyle\prod_{Z\in\sS(\Sigma_f)}\displaystyle\holim_{U\in \disk_k^{B/}}\displaystyle\holim_{V\in\disk_k^{\tZU}}
\Omega^\infty(V^\bullet\wedge \S)\ar[rrrr]_{\prod_{Z\in\sS(\Sigma_f)}\ind^d_\tZ}&&&&
\displaystyle\prod_{Z\in\sS(\Sigma_f)}\displaystyle\holim_{U\in \disk_k^{B/}}\displaystyle\holim_{V\in\disk_k^{\tZU}}
\Omega^\infty(V^\bullet\wedge \S)\ar[uuuu]_{+}
}
\end{adjustbox}
\end{theorem}

\begin{proof}
We begin by introducing an auxilliary map. Let $e^d_X(p)$ be the \textit{vector field Euler section}, a map 

$$
S^0\xrightarrow{e^d_X(p)} \displaystyle\holim_{U\in\disk_k^{B/}}\displaystyle\holim_{V\in\disk_m^{p^{-1}U/}}\Omega^\infty(V^\bullet\wedge\S)
$$

\noindent given locally as

\begin{equation}\label{eqn:vfieldchar_local_s}
 U^\bullet\wedge S^d\wedge S^0\xrightarrow{c\wedge \id}V^\bullet\wedge\Th(\nu|_V)\wedge S^0\xrightarrow{\id\wedge e_X^d(\tau)}V^\bullet\wedge\Th(\nu|_V)\wedge\Th(\tau|_V)\wedge \S
\end{equation}

The only difference between the composition above and~(\ref{eqn:char_map_d}) is the map $e^d_X(\tau)$, which we define to be the composition
$$
S^0\xrightarrow{e^d_n(\tau)}\Th(\tau|_V)\wedge\S\xrightarrow{\ell\wedge\id}\Th(\tau|_V)\wedge\S
$$
in which the map $\ell:\Th(\tau|_V)\to \Th(\tau|_V)$ sends a point over $y\in V$ to $y+X(y)$. By placing the scaling factor $t\in[0,1]$ as a coefficient in front of $X$, we obtain a family of maps $\ell_t$ which is a homotopy between the map $\ell$ and the identity. Thus, the vector field Euler section $e_X^d(p)$ is homotopic to the Euler section $e^d(p)$ from before. This results in homotopy commutativity of the following diagram:
\begin{equation}\label{eqn:vfieldchar}
\xymatrix{
S^0\ar[rrrr]^{e^d(p)}\ar[rrrrd]_{e^d_X(p)}&&&&\displaystyle\holim_{U\in \disk_k^{B/}}\displaystyle\holim_{V\in\disk_m^{p^{-1}U/}}
\Omega^\infty(V^\bullet\wedge \S)\\
&&&&\displaystyle\holim_{U\in \disk_k^{B/}}\displaystyle\holim_{V\in\disk_m^{p^{-1}U/}}
\Omega^\infty(V^\bullet\wedge \S)\ar[u]^{\id}
}
\end{equation}

\medskip

The remainder of this proof is organized in smaller pieces, each of which proves the homotopy commutativity of a triangle in the diagram below.

\begin{adjustbox}{width=\columnwidth,center}
\xymatrix{
S^0
\ar[dddddddddddddd]_{\prod_{Z\in\sS(\Sigma_f)}e^d(\pi_\tZ)}
\ar[rrrr]^{e^d(p)}
\ar@/_4pc/[ddddddddddddddrrrr]^{\prod_{Z\in\sS(\Sigma_f)}e^d_X(\pi_\tZ)}
\ar@/_2pc/[ddddddddrrrr]^{e^d_X(\psi)}
\ar@/_/[ddrrrr]_{e_X^d(p)}
&&&&\displaystyle\holim_{U\in \disk_k^{B/}}\displaystyle\holim_{V\in\disk_m^{p^{-1}U/}}
\Omega^\infty(V^\bullet\wedge \S)
\\
&&&(1)
\\
&&&&\displaystyle\holim_{U\in \disk_k^{B/}}\displaystyle\holim_{V\in\disk_m^{p^{-1}U/}}
\Omega^\infty(V^\bullet\wedge \S)
\ar[uu]_{\id} 
\\
&&&(2)
\\
\\
\\
\\
&&&(3)
\\
&&&&\displaystyle\holim_{U\in \disk_k^{B/}}\displaystyle\holim_{V\in\disk_m^{\psi^{-1}U/}}
\Omega^\infty(V^\bullet\wedge \S)
\ar[uuuuuu]_{\text{incl}}
\\
\\
\\
\\
&(4)
\\
\\
\displaystyle\prod_{Z\in\sS(\Sigma_f)}\displaystyle\holim_{U\in \disk_k^{B/}}\displaystyle\holim_{V\in\disk_k^{\tZU}}
\Omega^\infty(V^\bullet\wedge \S)\ar[rrrr]_{\prod_{Z\in\sS(\Sigma_f)}\ind^d_\tZ}&&&&
\displaystyle\prod_{Z\in\sS(\Sigma_f)}\displaystyle\holim_{U\in \disk_k^{B/}}\displaystyle\holim_{V\in\disk_k^{\tZU}}
\Omega^\infty(V^\bullet\wedge \S)
\ar[uuuuuu]_g
}
\end{adjustbox}

\medskip

\medskip

We will now prove the homotopy commutativity of each subtriangle in the diagram above, beginning from the top triangle and proceeding clockwise. We will define the morphisms in each triangle as needed.

\begin{enumerate}
    \medskip\item The top triangle relating maps $e^d(p)$ and $e^d_X(p)$ is identically diagram (\ref{eqn:vfieldchar}).
    
    \medskip\item The map $e_X^d(\psi)$ is defined locally on a pair $U\in \disk_k^{B/}$ and $V\in\disk_m^{\psi^{-1}U/}$ as the composition given in (\ref{eqn:vfieldchar_local_s}). Since the characteristics $e^d_X(p)$ and $e^d_X(\psi)$ have identical local definitions, it suffices to see that the composition (\ref{eqn:vfieldchar_local_s}) is nullhomotopic on objects of $\disk_m^{p^{-1}U/}$ that are not also objects in $\disk_m^{\psi^{-1}U/}$. Since $DZ$ is the unit disk bundle on $Z$, the vector field $X$ has length greater than 1 at any point $x\not\in DZ$. This means that the map $j:\Th(\tau|_V)\to\Th(\tau|_V)$ appearing in the definition of the vector field Euler section $e^d(\tau)$ is nullhomotopic on any $V\in\disk_m^{p^{-1}U/}$ that is not also an object in $\disk_m^{\psi^{-1}U/}$. Thus the map $e_X^d(p)$ factors through $e_X^d(\psi)$ as indicated in the diagram above.
    
    \medskip\item The map $e_X^d(\pi_\tZ)$ is defined as the composition $\ind_X^d\circ e^d(\pi_\tZ)$ so that the bottom triangle commutes.
    
    \medskip\noindent Consider the functor $R_\tZ:\disk_m^{\psi^{-1}U/}\to\disk_k^{\tZU}$ that sends a disk $V$ corresponding to an embedding $\R^k\times \R^n\hookrightarrow D\Sigma_f$ to the intersection of the image of this embedding and $\tZ$. The map $g$ is induced by the product of functors $R_\tZ$ over $Z\in \sS(\Sigma_f)$.

    \medskip\noindent To see that this triangle commutes, it suffices to see that the local definitions of the characteristics $e_X^d(\pi_\tZ)$ and $e_X^d(\psi)$ agree up to homotopy. There are two cases, either $U$ intersects a bifurcation set in $B$, or it does not. Assuming that it does not, we have $U\in\disk_k^{B/}$, $V\in\disk_m^{\psi^{-1}U/}$, and $R_\tZ(V)\in \disk_k^{\tZU}$ for some $Z\in\sS(\Sigma_f)$, and we must check that the local definitions of the characteristics $e_X^d(\pi_\tZ)$ and $e_X^d(\psi)$ agree up to homotopy. Thus we must construct a homotopy that makes the diagram below commute. This diagram is obtained by comparing (\ref{eqn:vfieldchar_local_s}) applied to $V$ (the left vertical composition) with (\ref{eqn:char_map_d_Z}) composed with (\ref{eqn:index_local_d}) applied to $R_\tZ(V)$ (the composition along the top).

    \begin{adjustbox}{width=\columnwidth,center}
        \xymatrix{
        U^\bullet\wedge S^d\wedge S^0
        \ar[r]^{c\wedge \id}
        \ar[d]^{c\wedge \id}
        &
        R_\tZ(V)^\bullet\wedge\Th(\eta\oplus\nu|_{R_\tZ(V)})\wedge S^0
        \ar[dr]^{\id\wedge e_n^d(\tau)}
        \\
        V^\bullet\wedge\Th(\nu|_V)\wedge S^0
        \ar[d]^{\id\wedge e_X^d(\tau)}
        &
        &
        R_\tZ(V)^\bullet\wedge\Th(\eta\oplus\nu|_{R_\tZ(V)})\wedge\Th(\tau|_{R_\tZ(V)})\wedge \S
        \ar[d]^{\id\wedge\id\wedge dX|_{R_\tZ(V)}\wedge\id}
        \\
        V^\bullet\wedge\Th(\nu|_V)\wedge\Th(\tau|_V)\wedge \S
        &&
        R_\tZ(V)^\bullet\wedge\Th(\eta|_{R_\tZ(V)})\wedge\Th(\nu|_{R_\tZ(V)})\wedge\Th(\tau|_{R_\tZ(V)})\wedge \S\ar[ll]^{\text{incl}}
        }
    \end{adjustbox}
    
    \medskip\noindent The map $R_\tZ(V)^\bullet\wedge\Th(\eta|_{R_\tZ(V)})\to V^\bullet$ is homotopic to the identity. Recall that the map $e^d_X(\tau)$ is defined to be the composition
    $$
    S^0\xrightarrow{e^d_n(\tau)}\Th(\tau|_V)\wedge\S\xrightarrow{j\wedge\id}\Th(\tau|_V)\wedge\S
    $$
    Thus it remains to show that the map $dX|_{R_\tZ(V)}$ is homotopic to $j$. However, this is an immediate consequence of the fact that in the neighborhood of a nondegenerate zero a vector field is homotopic to its derivative.

    \medskip\noindent Next we assume that $U$ does intersect a bifurcation set. 
    Then for some $V\in \disk_m^{\psi^{-1}U/}$ we have nontrivial $R_{\tZ_i}(V)\in \disk_k^{\tZ_i/U}$ and $R_{\tZ_{i+1}}(V)\in \disk_k^{\tZ_{i+1}/U}$, and we must see that the local definition of the characteristic $e_X^d(\psi)$ agrees with the wedge sum of the local definitions of the characteristics $e_X^d(\pi_{\tZ_i})$ and $e_X^d(\pi_{\tZ_{i+1}})$. 
    It suffices to verify that the map $dX|_{R_{\tZ_i}(V)}\vee dX|_{R_{\tZ_{i+1}}(V)}$ composed with the inclusion 
    
    $$
    \Th(\tau|_{R_{\tZ_i}(V)})\vee \Th(\tau|_{R_{\tZ_{i+1}(V)}})\to\Th(\tau|_V)
    $$
    
    \noindent is homotopic to $j:\Th(\tau|_V)\to \Th(\tau|_V)$. This homotopy is the homotopy of the vector field $X$ with only one degenerate zero to a vector field with two nondegenerate zeros of degree $i$ and $i+1$, which is obtained by taking the time varying gradient vector field of the family of functions in Proposition~\ref{prop:gmf_normal_form}.

    \medskip\item Recall from the previous step that the map $e_X^d(\pi_\tZ)$ is defined as the composition $\ind_X^d\circ e^d(\pi_\tZ)$ so that the bottom triangle commutes.
    
\end{enumerate}

\end{proof}

\subsection{Fiberwise Poincar\'e--Hopf Theorem for the excisive A-theory Euler characteristic}
\label{subsec:AthyPH}

In this section we prove a fiberwise Poincar\'e–Hopf theorem for the excisive A-theory Euler characteristic. The proof of this result is largely the same as the proof of Theorem~\ref{thm:PHQthy} in the previous section, with the main differences being that the Euler section $e^d_n$ is replaced with $e^t_n$, and the coefficient spectrum $\S$ is replaced with $A(*)$. All indexing categories used in this section are defined in Section~\ref{subsec:QthyPH}. However, we give new definitions of the maps and an abbreviated proof to keep the discussion in this section mostly self contained. 

\subsubsection{Preliminaries}

\medskip\noindent We begin by defining an Euler section and index map associated to $Z$. To construct the characteristic $e^t(\pi_{\widetilde{Z}})$ we consider the following composition analogous to (\ref{eqn:char_map_d_Z}) from the previous section. 

\begin{equation}\label{eqn:char_map_t_Z}
    U^\bullet\wedge S^d\wedge S^0\xrightarrow{c\wedge \id}V^\bullet\wedge\Th(\eta\oplus\nu|_V)\wedge S^0\xrightarrow{\id\wedge e_n^d(\tau)}V^\bullet\wedge\Th(\eta\oplus\nu|_V)\wedge\Th(\tau|_V)\wedge \S
\end{equation}

This induces a map

$$
S^0\xrightarrow{e^t(\pi_{\widetilde{Z}})} \displaystyle\holim_{U\in\disk_k^{B/}}\displaystyle\holim_{V\in\disk_k^{\tZU}}\Omega^\infty(V^\bullet\wedge A(*))
$$

\noindent Aggregate over all $Z\in\sS(\Sigma_f)$, we have a map

$$
S^0\xrightarrow{\prod_{Z\in\sS(\Sigma_f)}e^t(\pi_\tZ)} \displaystyle\prod_{Z\in\sS(\Sigma_f)}\displaystyle\holim_{U\in\disk_k^{B/}}\displaystyle\holim_{V\in\disk_k^{\tZU}}\Omega^\infty(V^\bullet\wedge A(*))
$$

\medskip To define the index map, we once again consider the derivative of the gradient vector field $X$ at $z\in Z$. This is a map $dX_z:\tau_z\to\tau_z$ which induces a self-map on the one point compactification of $\tau_z$, and thus a map $dX:\Th(\tau_Z)\to\Th(\tau_Z)$. For any $V\in\disk_k^{Z/}$, we can restrict the map $dX$ to $V$ to obtain a map $dX|_V:\Th(\tau|_V)\to\Th(\tau|_V)$. Then the local map 

\begin{equation}\label{eqn:index_local}
V^\bullet\wedge\Th(\nu|_V)\wedge\Th(\tau|_V)\wedge A(*)\xrightarrow{\id\wedge\id\wedge dX|_V\wedge\id}
V^\bullet\wedge\Th(\nu|_V)\wedge\Th(\tau|_V)\wedge A(*)
\end{equation}

\medskip\noindent induces a map

\begin{equation}\label{eqn:t_index_map}
\displaystyle\holim_{U\in\disk_k^{B/}}\displaystyle\holim_{V\in\disk_k^{\tZU}}\Omega^\infty(V^\bullet\wedge A(*))\xrightarrow{\ind^t_Z}
\displaystyle\holim_{U\in\disk_k^{B/}}\displaystyle\holim_{V\in\disk_k^{\tZU}}\Omega^\infty(V^\bullet\wedge A(*))
\end{equation}

\noindent which we denote as $\ind_Z^t$ and refer to as the index map on $Z$ with coefficients in $A(*)$. 

\noindent Aggregate over all $Z\in \sS(\Sigma_f)$ we have a map

\begin{equation}\label{eqn:t_index_map}
\displaystyle\prod_{Z\in\sS(\Sigma_f)}\displaystyle\holim_{U\in\disk_k^{B/}}\displaystyle\holim_{V\in\disk_k^{\tZU}}\Omega^\infty(V^\bullet\wedge A(*))\xrightarrow{\prod_{Z\in\sS(\Sigma_f)}\ind^t_\tZ}
\displaystyle\prod_{Z\in\sS(\Sigma_f)}\displaystyle\holim_{U\in\disk_k^{B/}}\displaystyle\holim_{V\in\disk_k^{\tZU}}\Omega^\infty(V^\bullet\wedge A(*))
\end{equation}

\noindent The following lemma will be used in Section~\ref{subsec:PHSD}.

\begin{lemma}\label{lem:ind_t_on_htpy_gps}
For $Z\in\sS(\Sigma_f)$ of degree $j$, $\ind_\tZ^t$ is multiplication by $(-1)^j$ on homotopy groups. 
\end{lemma}
\begin{proof}
This follows immediately from Lemma~\ref{lem:ind_on_htpy_gps}, since the definition of $\ind_\tZ^t$ is identical to that of $\ind_\tZ^d$, up to a change of coefficients on which the index map is the identity. 
\end{proof}

\subsubsection{Factoring the excisive A-theory Euler characteristic}

\begin{theorem}\label{thm:PHAthy}
The diagram below is homotopy commutative.

\begin{adjustbox}{width=\columnwidth,center}
\xymatrix{
S^0\ar[dddd]_{\displaystyle\prod_{Z\in\sS(\Sigma_f)}e^t(\pi_\tZ)}\ar[rrrr]^{e^t(p)}
&&&&
\displaystyle\holim_{U\in \disk_k^{B/}}\displaystyle\holim_{V\in\disk_m^{p^{-1}U/}}
\Omega^\infty(V^\bullet\wedge A(*))\\
\\
\\
\\
\displaystyle\prod_{Z\in\sS(\Sigma_f)}\displaystyle\holim_{U\in \disk_k^{B/}}\displaystyle\holim_{V\in\disk_k^{\tZU}}
\Omega^\infty(V^\bullet\wedge A(*))\ar[rrrr]_{\displaystyle\prod_{Z\in\sS(\Sigma_f)}\ind^t_\tZ}
&&&&
\displaystyle\prod_{Z\in\sS(\Sigma_f)}\displaystyle\holim_{U\in \disk_k^{B/}}\displaystyle\holim_{V\in\disk_k^{\tZU}}
\Omega^\infty(V^\bullet\wedge A(*))\ar[uuuu]_{+}
}
\end{adjustbox}
\end{theorem}
\begin{proof}
In this proof we rehash the main steps of the proof of Theorem~\ref{thm:PHQthy}, making the necessary changes as needed. 

\medskip\noindent We begin by defining the vector field Euler section $e_X^t(p)$, a map

$$
S^0\xrightarrow{e^t_X(p)} \displaystyle\holim_{U\in\disk_k^{B/}}\displaystyle\holim_{V\in\disk_m^{p^{-1}U/}}\Omega^\infty(V^\bullet\wedge A(*))
$$

\noindent given locally as

\begin{equation}\label{eqn:vfieldchar_local}
 U^\bullet\wedge S^d\wedge S^0\xrightarrow{c\wedge \id}V^\bullet\wedge\Th(\nu|_V)\wedge S^0\xrightarrow{\id\wedge e_X^t(\tau)}V^\bullet\wedge\Th(\nu|_V)\wedge\Th(\tau|_V)\wedge A(*)
\end{equation}

\medskip\noindent The map $e^t_X(\tau)$ is defined to be the composition
$$
S^0\xrightarrow{e^d_n(\tau)}\Th(\tau|_V)\wedge A(*)\xrightarrow{\ell\wedge\id}\Th(\tau|_V)\wedge A(*)
$$
in which the map $\ell:\Th(\tau|_V)\to \Th(\tau|_V)$ sends a point over $y\in V$ to $y+X(y)$. By placing the scaling factor $t\in[0,1]$ as a coefficient in front of $X$, we obtain a family of maps $\ell_t$ which is a homotopy between the map $\ell$ and the identity. Thus, the vector field Euler section $e_X^t(p)$ is homotopic to the Euler section $e^t(p)$. This results in homotopy commutativity of the following diagram:

\begin{equation}\label{eqn:vfieldchar_t}
\xymatrix{
S^0\ar[rrrr]^{e^t(p)}\ar[rrrrd]_{e^t_X(p)}&&&&\displaystyle\holim_{U\in \disk_k^{B/}}\displaystyle\holim_{V\in\disk_m^{p^{-1}U/}}
\Omega^\infty(V^\bullet\wedge A(*))\\
&&&&\displaystyle\holim_{U\in \disk_k^{B/}}\displaystyle\holim_{V\in\disk_m^{p^{-1}U/}}
\Omega^\infty(V^\bullet\wedge A(*))\ar[u]^{\id}
}
\end{equation}

\noindent The remainder of this proof is organized in smaller pieces, each of which proves the homotopy commutativity of a triangle in the diagram below. 

\medskip

\begin{adjustbox}{width=\columnwidth,center}

\xymatrix{
S^0
\ar[dddddddddddddd]_{\prod_{Z\in\sS(\Sigma_f)}e^t(\pi)}
\ar[rrrr]^{e^t(p)}
\ar@/_4pc/[ddddddddddddddrrrr]^{\prod_{Z\in\sS(\Sigma_f)}e^t_X(\pi)}
\ar@/_2pc/[ddddddddrrrr]^{e^t_X(\psi)}
\ar@/_/[ddrrrr]_{e_X^t(p)}
&&&&\displaystyle\holim_{U\in \disk_k^{B/}}\displaystyle\holim_{V\in\disk_m^{p^{-1}U/}}
\Omega^\infty(V^\bullet\wedge A(*))
\\
&&&(1)
\\
&&&&\displaystyle\holim_{U\in \disk_k^{B/}}\displaystyle\holim_{V\in\disk_m^{p^{-1}U/}}
\Omega^\infty(V^\bullet\wedge A(*))
\ar[uu]_{\id} 
\\
\\
&&&(2)
\\
\\
\\
&&&(3)
\\
&&&&\displaystyle\holim_{U\in \disk_k^{B/}}\displaystyle\holim_{V\in\disk_m^{\psi^{-1}U/}}
\Omega^\infty(V^\bullet\wedge A(*))
\ar[uuuuuu]_{\text{incl}}
\\
\\
\\
\\
&(4)
\\
\\
\displaystyle\prod_{Z\in\sS(\Sigma_f)}\displaystyle\holim_{U\in \disk_k^{B/}}\displaystyle\holim_{V\in\disk_k^{\tZU}}
\Omega^\infty(V^\bullet\wedge A(*))\ar[rrrr]_{\prod_{Z\in\sS(\Sigma_f)}\ind^t_Z}&&&&
\displaystyle\prod_{Z\in\sS(\Sigma_f)}\displaystyle\holim_{U\in \disk_k^{B/}}\displaystyle\holim_{V\in\disk_k^{\tZU}}
\Omega^\infty(V^\bullet\wedge A(*))
\ar[uuuuuu]_g^\simeq
}
\end{adjustbox}

\medskip

\medskip

The map $e_X^t(\pi)$ in the diagram above is defined as the composition $\ind_Z^t\circ e^t(\pi)$, hence the bottom triangle commutes. The triangle at the top of the diagram above is diagram (\ref{eqn:vfieldchar_t}). The proofs of the commutativity of the middle two triangles are identical to the analogous steps in the proof of Theorem~\ref{thm:PHQthy}, so long as the Euler sections $e^d_n(\tau)$ are replaced with $e^t_n(\tau)$, and the coefficient spectrum $\S$ is replaced with $A(*)$. 
\end{proof}

We conclude this section by giving a generalization of Theorems~\ref{thm:PHQthy} and~\ref{thm:PHAthy}. 

\begin{theorem}\label{thm:PHcube}
The following diagram of spaces is homotopy commutative.

\medskip 

\medskip 

\begin{adjustbox}{width=\columnwidth,center}
\begin{tikzcd}[column sep={12em,between origins},row sep=4em]
&
S^0
  \arrow[rr,"e^d(p)"]
  \arrow[dl,swap,"\prod_{Z\in\sS(\Sigma_f)}e^d(\pi)"]
  \arrow[dd,"\id"]
&&
\displaystyle\holim_{U\in \disk_k^{B/}}\displaystyle\holim_{V\in\disk_m^{p^{-1}U/}}
\Omega^\infty(V^\bullet\wedge \S)  
\arrow[dd,"\eta"]
\\
\displaystyle\prod_{Z\in\sS(\Sigma_f)}\displaystyle\holim_{U\in \disk_k^{B/}}\displaystyle\holim_{V\in\disk_k^{\tZU}}
\Omega^\infty(V^\bullet\wedge \S)
  \arrow[rr,crossing over,swap,"\prod_{Z\in\sS(\Sigma_f)}\ind^d_\tZ"]
  \arrow[dd,"\eta"]
&&
\displaystyle\prod_{Z\in\sS(\Sigma_f)}\displaystyle\holim_{U\in \disk_k^{B/}}\displaystyle\holim_{V\in\disk_k^{\tZU}}
\Omega^\infty(V^\bullet\wedge \S)
  \arrow[ur,"+"]
  \arrow[dd,crossing over,"\eta"]
\\
&
S^0
  \arrow[rr,"e^t(p)"]
  \arrow[dl,swap,"\prod_{Z\in\sS(\Sigma_f)}e^t(\pi)"]
&&
\displaystyle\holim_{U\in \disk_k^{B/}}\displaystyle\holim_{V\in\disk_m^{p^{-1}U/}}
\Omega^\infty(V^\bullet\wedge A(*))
\\
\displaystyle\prod_{Z\in\sS(\Sigma_f)}\displaystyle\holim_{U\in \disk_k^{B/}}\displaystyle\holim_{V\in\disk_k^{\tZU}}
\Omega^\infty(V^\bullet\wedge A(*))
\arrow[rr,swap,"\prod_{Z\in\sS(\Sigma_f)}\ind^t_Z"]
&&
\displaystyle\prod_{Z\in\sS(\Sigma_f)}\displaystyle\holim_{U\in \disk_k^{B/}}\displaystyle\holim_{V\in\disk_k^{\tZU}}
\Omega^\infty(V^\bullet\wedge A(*))
\arrow[ur,"+"]
\end{tikzcd}
\end{adjustbox}

\end{theorem}

\medskip

\begin{proof}
Theorem~\ref{thm:PHQthy} indicates that the top square commutes, and Theorem~\ref{thm:PHAthy} indicates that the bottom square commutes. The back face of the cube commutes by Proposition~\ref{prop:DWW_vanishing}. 

\medskip

Homotopy commutativity of each of the remaining vertical faces follows from applying Theorem 4.13 in~\cite{DWW03} to the local definitions of each map to construct the homotopy locally on each disk. 
\end{proof}

\subsection{Fiberwise Poincar\'e–Hopf Theorem for stratified deformations}
\label{subsec:PHSD}

In this section, our primary goal is to generalize the definitions of Sections \ref{subsec:QthyPH} and \ref{subsec:AthyPH} to arbitrary stratified subsets, so that we may generalize Theorems \ref{thm:PHQthy}, \ref{thm:PHAthy}, and \ref{thm:PHcube} to the particular stratified deformation constructed in Section \ref{subsec:strat_def}. 

We begin by fixing some notation for this section. Let $(\Sigma,\psi)$ denote an arbitrary stratified subset of $p:W\to B$ with coefficients in $X=BO\times BO$. The map $\psi:\Sigma\to BO\times BO$ composed with projection onto the first factor is denoted $\gamma_-^\Sigma$, and when composed with projection onto the second factor is denoted $\gamma_+^\Sigma$. These choices will ultimately be used to classify the negative and positive eigenspace bundles on $\Sigma$, as in Example \ref{ex:GMF_SD}. Let $\sS(\Sigma)$ denote the collection of strata in the degree-wise stratification on $\Sigma$. We denote an element of $\sS(\Sigma)$ in degree $i$ by $Z_i$, but may drop the subscript when it is unnecessary. Stratified subsets admit ghosts just the same as the critical locus of a fiberwise generalized Morse function. As before, we use the notation $\tZ$ to denote the smooth manifold obtained by perturbing a stratum of $\sS(\Sigma)$ over the ghost set as in Section \ref{subsec:ghosts}. 

\medskip\noindent For $Z\in\sS(\Sigma)$, the category $\disk_m^\tZU$ is defined identically as in Definition \ref{defn:diskZ_U}. As in Sections \ref{subsec:QthyPH} and \ref{subsec:AthyPH}, the composition~(\ref{eqn:char_map_d_Z}) is used to construct a map

$$
S^0\xrightarrow{e^d(\pi_\tZ)} \displaystyle\holim_{U\in\disk_k^{B/}}\displaystyle\holim_{V\in\disk_k^{\tZU}}\Omega^\infty(V^\bullet\wedge\S)
$$

\noindent and the composition~(\ref{eqn:char_map_t_Z}) is used to construct a map 

$$
S^0\xrightarrow{e^t(\pi_\tZ)} \displaystyle\holim_{U\in\disk_k^{B/}}\displaystyle\holim_{V\in\disk_k^{\tZU}}\Omega^\infty(V^\bullet\wedge A(*))
$$

\noindent Aggregate over all $Z\in\sS(\Sigma)$, we also have the maps

\begin{equation*}
S^0\xrightarrow{\prod_{Z\in\sS(\Sigma)}e^d(\pi_\tZ)} \displaystyle\prod_{Z\in\sS(\Sigma_f)}\displaystyle\holim_{U\in\disk_k^{B/}}\displaystyle\holim_{V\in\disk_k^{\tZU}}\Omega^\infty(V^\bullet\wedge \S)
\end{equation*}

\begin{equation*}
S^0\xrightarrow{\prod_{Z\in\sS(\Sigma)}e^t(\pi_\tZ)} \displaystyle\prod_{Z\in\sS(\Sigma_f)}\displaystyle\holim_{U\in\disk_k^{B/}}\displaystyle\holim_{V\in\disk_k^{\tZU}}\Omega^\infty(V^\bullet\wedge A(*))
\end{equation*}

\noindent In particular, the definitions of the local characteristics on $\Sigma$ do not depend on $\Sigma$ being obtained as the critical locus of a fiberwise generalized Morse function.

\medskip

Next, we must generalize the definitions of the index maps from before. Recall that when working with the critical locus of a fiberwise generalized Morse function, the index map was defined using a family of matrices which provided an automorphism of the one point compactification of the vertical tangent space at each point in the critical locus. That particular matrix was the derivative of the vertical gradient vector field of the function, or the Hessian of the function. 

For an arbitrary stratified deformation, our matrix is given by the block sum of negative the identity matrix on the negative eigenspace bundle, and the identity matrix on the positive eigenspace bundle. In the event that $\Sigma$ is the critical locus of a fiberwise generalized Morse function, this choice clearly agrees with the maps (\ref{eqn:d_index_map}) and (\ref{eqn:t_index_map}). Thus, we use the same notation for index maps on arbitrary stratified subsets: $\ind_\tZ^d$ and $\ind_\tZ^h$. 

Let $(\Sigma_f,\psi_f)$ be the stratified subset obtained from the critical locus of a fiberwise generalized Morse function. Let $(S,\Psi)$ be a stratified deformation between stratified subsets $(\Sigma_f,\psi_f)$ and $(\Sigma_\sd,\psi_\sd)$. For the stratified subset $(\Sigma_\sd,\psi_\sd)$, consider the composition below

\begin{equation}
\xymatrix{
S^0\ar[dd]_{\prod_{Z\in\sS(\Sigma_\sd)}e^d(\pi_{\tZ})}&&\displaystyle\holim_{U\in \disk_k^{B/}}\displaystyle\holim_{V\in\disk_m^{p^{-1}U/}}
\Omega^\infty(V^\bullet\wedge \S)\\
\\
\displaystyle\prod_{Z\in\sS(\Sigma_\sd)}\displaystyle\holim_{U\in\disk_k^{B/}}\displaystyle\holim_{V\in\disk_k^{\tZ_l/U}}\Omega^\infty(V^\bullet\wedge \S)
\ar[rr]^{\prod_{Z\in\sS(\Sigma_\sd)}\ind_\tZ^d}&& \displaystyle\prod_{Z\in\sS(\Sigma_\sd)}\displaystyle\holim_{U\in\disk_k^{B/}}\displaystyle\holim_{V\in\disk_k^{\tZ/U}}\Omega^\infty(V^\bullet\wedge \S)
\ar[uu]_{+}
}
\end{equation}

Because $(\Sigma_f,\psi_f)$ and $(\Sigma_\sd,\psi_\sd)$ form the boundaries of a stratified subset $(S,\Psi)$ in $p\times I:W\times I\to B\times I$, the compositions below are homotopic. 

\begin{equation}\label{eqn:SD_htpy}
\left(
\left(\prod_{Z\in\sS(\Sigma_f)}e^d(\pi_{\tZ})\right)
\circ\left(
\prod_{Z\in\sS(\Sigma_f)}\ind_\tZ^d
\right)
\circ +
\right)
\sim
\left(
\left(\prod_{Z\in\sS(\Sigma_\sd)}e^d(\pi_{\tZ_1})\right)
\circ\left(
\prod_{Z\in\sS(\Sigma_\sd)}\ind_\tZ^d
\right)
\circ +
\right)
\end{equation}

\noindent Thus, we have the following theorem.

\newpage
\begin{theorem}\label{thm:PHHSD}
If there exists $(S,\Psi)$ a stratified deformation between $(\Sigma_f,\psi_f)$ and $(\Sigma_\sd,\psi_\sd)$, then the following diagram is homotopy commutative. 

\begin{adjustbox}{width=\columnwidth,center}
\xymatrix{
S^0\ar[dddd]_{\prod_{Z\in\sS(\Sigma_\sd)}e^d(\pi_\tZ)}\ar[rrrr]^{e^d(p)}&&&&\displaystyle\holim_{U\in \disk_k^{B/}}\displaystyle\holim_{V\in\disk_m^{p^{-1}U/}}
\Omega^\infty(V^\bullet\wedge \S)\\
\\
\\
\\
\displaystyle\prod_{Z\in\sS(\Sigma_\sd)}\displaystyle\holim_{U\in \disk_k^{B/}}\displaystyle\holim_{V\in\disk_k^{\tZU}}
\Omega^\infty(V^\bullet\wedge \S)\ar[rrrr]_{\prod_{Z\in\sS(\Sigma_\sd)}\ind^d_\tZ}&&&&
\displaystyle\prod_{Z\in\Sigma_\sd}\displaystyle\holim_{U\in \disk_k^{B/}}\displaystyle\holim_{V\in\disk_k^{\tZU}}
\Omega^\infty(V^\bullet\wedge \S)\ar[uuuu]_{{+}}
}
\end{adjustbox}
\end{theorem}
\begin{proof}
First we apply Theorem~\ref{thm:PHQthy} to $(\Sigma_f,\psi_f)$. Then the homotopy~(\ref{eqn:SD_htpy}) is used to make the diagram above commute. 
\end{proof}

\begin{rmk}
The homotopy (\ref{eqn:SD_htpy}) used in the proof of Theorem~\ref{thm:PHHSD} indicates that the composition in question is a \textit{stratified deformation invariant}. Compare to~\cite{Igu05} Lemma 5.4.
\end{rmk}

\medskip\noindent We now have the following generalization of Theorem~\ref{thm:PHcube}:

\begin{theorem}\label{thm:PHcubeSD}
The following diagram of spaces is homotopy commutative.

\medskip 

\medskip 

\begin{adjustbox}{width=\columnwidth,center}
\begin{tikzcd}[column sep={12em,between origins},row sep=4em]
&
S^0
  \arrow[rr,"e^d(p)"]
  \arrow[dl,swap,"\prod_{Z\in\sS(\Sigma_\sd)}e^d(\pi)"]
  \arrow[dd,"\id"]
&&
\displaystyle\holim_{U\in \disk_k^{B/}}\displaystyle\holim_{V\in\disk_m^{p^{-1}U/}}
\Omega^\infty(V^\bullet\wedge \S)  
\arrow[dd,"\eta"]
\\
\displaystyle\prod_{Z\in\sS(\Sigma_\sd)}\displaystyle\holim_{U\in \disk_k^{B/}}\displaystyle\holim_{V\in\disk_k^{\tZU}}
\Omega^\infty(V^\bullet\wedge \S)
  \arrow[rr,crossing over,swap,"\prod_{Z\in\sS(\Sigma_\sd)}\ind^d_\tZ"]
  \arrow[dd,"\eta"]
&&
\displaystyle\prod_{Z\in\sS(\Sigma_\sd)}\displaystyle\holim_{U\in \disk_k^{B/}}\displaystyle\holim_{V\in\disk_k^{\tZU}}
\Omega^\infty(V^\bullet\wedge \S)
  \arrow[ur,"+"]
  \arrow[dd,crossing over,"\eta"]
\\
&
S^0
  \arrow[rr,"e^t(p)"]
  \arrow[dl,"\prod_{Z\in\sS(\Sigma_\sd)}e^t(\pi)"]
&&
\displaystyle\holim_{U\in \disk_k^{B/}}\displaystyle\holim_{V\in\disk_m^{p^{-1}U/}}
\Omega^\infty(V^\bullet\wedge A(*))
\\
\displaystyle\prod_{Z\in\sS(\Sigma_\sd)}\displaystyle\holim_{U\in \disk_k^{B/}}\displaystyle\holim_{V\in\disk_k^{\tZU}}
\Omega^\infty(V^\bullet\wedge A(*))
\arrow[rr,"\prod_{Z\in\sS(\Sigma_\sd)}\ind^t_Z"]
&&
\displaystyle\prod_{Z\in\sS(\Sigma_\sd)}\displaystyle\holim_{U\in \disk_k^{B/}}\displaystyle\holim_{V\in\disk_k^{\tZU}}
\Omega^\infty(V^\bullet\wedge A(*))
\arrow[ur,"+"]
\end{tikzcd}
\end{adjustbox}

\end{theorem}

\subsubsection{Rational fiberwise Poincar\'e–Hopf formulas}

In Subsection~\ref{subsec:PHTheta} and Section~\ref{sec:ThetaCalc}, we will make use of the following simplifications of Theorem~\ref{thm:PHcubeSD}. Let $\Sigma_\sd^\#$ be the complement of the component $\Lambda$ of $\Sigma_\sd$ on which $\psi_\sd$ is trivial (see Lemmas \ref{lem:SD2} and \ref{lem:null_deformable}).  Let $\mathscr{A}_j\subset\sS(\Sigma^\#_\sd)$ contain those elements of degree $j$ corresponding to the lower stratum of the remaining immersed lenses and let $\mathscr{A}_{j+1}\subset\sS(\Sigma^\#_\sd)$ contain those elements of degree $j+1$ corresponding to the upper stratum of the immersed lenses. We also denote by $Z_i\in\sS(\Sigma_\sd)$ a stratum in degree $i$.

\begin{corollary}\label{cor:ratl_formula_d}
The following equality holds in $\pi_0\Gamma_B Q_B(W)\otimes\Q$
$$
e^d_\partial(p) = \displaystyle\sum_{Z_i\in\sS(\Sigma^\#_\sd)} (-1)^i e^d(\pi_{\tZ_i}) 
$$
\end{corollary}
\begin{proof}
In $\pi_0\Gamma_B Q_B(W)$, Theorem~\ref{thm:PHHSD} reduces to the formula 
$$
e^d_\partial(p) = \displaystyle\sum_{Z_i\in\sS(\Sigma_\sd)} \ind^d_\tZ e^d(\pi_{\tZ_i}) 
$$
By Lemma~\ref{lem:ind_on_htpy_gps}, we can replace $\ind^d_\tZ$ with $(-1)^i$ as in the formula below
$$
e^d_\partial(p) = \displaystyle\sum_{Z_i\in\sS(\Sigma_\sd)} (-1)^i e^d(\pi_{\tZ_i}) 
$$
By Lemma~\ref{lem:null_deformable}, we can eliminate the summands for those $Z$ in $\sS(\Lambda)$, and we are left with the following formula in $\pi_0\Gamma_B Q_B(W)\otimes\Q$. 
$$
e^d_\partial(p) = \displaystyle\sum_{Z_i\in\sS(\Sigma^\#_\sd)} (-1)^i e^d(\pi_{\tZ_i}) 
$$
\end{proof}

\begin{corollary}\label{cor:ratl_formula_t}
The following equality holds in $\pi_0\Gamma_B A^\%_B(W)\otimes\Q$
$$
e^t_\partial(p) = \displaystyle\sum_{Z_i\in\sS(\Sigma^\#_\sd)} (-1)^i e^t(\pi_{\tZ_i}) 
$$
\end{corollary}
\begin{proof}
In $\pi_0\Gamma_B A^\%_B(W)$, Theorem~\ref{thm:PHHSD} reduces to the formula 
$$
e^t_\partial(p) = \displaystyle\sum_{Z_i\in\sS(\Sigma_\sd)} \ind^d_\tZ e^t(\pi_{\tZ_i}) 
$$
By Lemma~\ref{lem:ind_t_on_htpy_gps}, we can replace $\ind^d_\tZ$ with $(-1)^i$ as in the formula below
$$
e^t_\partial(p) = \displaystyle\sum_{Z_i\in\sS(\Sigma_\sd)} (-1)^i e^t(\pi_{\tZ_i}) 
$$
By Lemma~\ref{lem:null_deformable}, we can eliminate the summands for those $Z$ in $\sS(\Lambda)$, and we are left with the following formula in $\pi_0\Gamma_B A^\%_B(W)\otimes\Q$. 
$$
e^t_\partial(p) = \displaystyle\sum_{Z_i\in\sS(\Sigma^\#_\sd)} (-1)^i e^t(\pi_{\tZ_i}) 
$$
\end{proof}

\subsection{Fiberwise Poincar\'e–Hopf theorem for the smooth structure characteristic}
\label{subsec:PHTheta}

In this section, we return to the setting of Section \ref{sec:2}, in which we have a topologically trivial family of smooth h-cobordisms $p:W\to B$ with boundaries $\partial_0W:=M$ and $\partial_1W:= M'$ given as smooth manifold bundles $p_0:M\to B$ and $p_1:M'\to B$. The bundle $p$ admits a fiberwise generalized Morse function $f:W\to [0,1]$. By Theorem \ref{thm:DWWpullback}, we have a canonical nullhomotopy of the relative excisive A-theory Euler characteristic $\chi^\%_\partial(p)$. It then follows by Proposition \ref{prop:DWW_index_thms} that we have a nullhomotopy of the Euler section $e^t_\partial(p)$. This nullhomotopy is used in the following definition.

\begin{defn}\label{defn:euler_H}
For $p$ a topologically trivial family of smooth h-cobordisms $p:W\to B$, the Euler section $e^{t/d}_\partial(p)$ is a map 
$$
S^0\xrightarrow{{e^{t/d}_\partial(p)}}\displaystyle\holim_{U\in \disk_k^{B/}}\displaystyle\holim_{V\in\disk_m^{p^{-1}U/}}
\Omega^\infty(V^\bullet\wedge \h(*))
$$
defined to be the lift of $e^d_\partial(p)$ obtained from the nullhomotopy of $e^t_\partial(p)$. 
\end{defn}

In the proof of the theorem below, we express this nullhomotopy in terms of $\Sigma_\sd$, the stratified subset from Subsection \ref{subsec:strat_def}. Recall that the stratified subset is concentrated in two degrees, and is obtained by applying a stratified deformation to the critical locus of $f$. 

\medskip For the remainder of this paper, we denote $e^{t/d}(\pi_\tZ)\in\pi_0\Gamma_B\h^\%_B(W)\otimes\Q$ to be the lift of $e^d(\pi_\tZ)$ resulting from the identity $\eta_* e^d(\pi_\tZ)=e^t(\pi_\tZ)=0$.

\begin{theorem}\label{thm:PHTheta}
For $p:W\to B$ a topologically trivial family of smooth h-cobordisms, the following equality holds in $\pi_0\Gamma_B\h^\%_B(W)\otimes\Q$
$$
e^{t/d}_\partial(p) = \displaystyle\sum_{Z_i\in\sS(\Sigma^\#_\sd)} (-1)^i e^{t/d}(\pi_{\tZ_i})
$$
\end{theorem}
\begin{proof}
Corollary \ref{cor:ratl_formula_d} indicates that the following equality holds in $\pi_0\Gamma_B Q_B(W)\otimes\Q$
$$
e^d_\partial(p) = \displaystyle\sum_{Z_i\in\sS(\Sigma^\#_\sd)} (-1)^i e^d(\pi_{\tZ_i}) 
$$
From Corollary \ref{cor:ratl_formula_t}, if we apply $\eta_*$ to both sides we have the following equality in $\pi_0\Gamma_B A^\%_B(W)\otimes\Q$
$$
e^t_\partial(p) = \displaystyle\sum_{Z_i\in\sS(\Sigma^\#_\sd)} (-1)^i e^t(\pi_{\tZ_i}) 
$$
Since $p$ is a topologically trivial family of smooth h-cobordisms, $e^t_\partial(p)=0$ in $\pi_0\Gamma_B A^\%_B(W)\otimes\Q$. Recall from Lemma \ref{lem:SD1} and Lemma \ref{lem:SD2} that the elements of $\sS(\Sigma^\#_\sd)$ are concentrated in two degrees, $j$ and $j+1$. Let $\mathscr{A}_j\subset\sS(\Sigma^\#_\sd)$ contain those elements of degree $j$ and let $\mathscr{A}_{j+1}\subset\sS(\Sigma^\#_\sd)$ contain those elements of degree $j+1$. Since $\psi_\sd$ is trivial on each element $Z_j$ in $\mathscr{A}_j$, for all such elements we have that $e^t(\pi_{\tZ_j})=0$. We are left with the equation
$$
\displaystyle\sum_{Z\in\mathscr{A}_{j+1}} e^t(\pi_{\tZ}) = 0 
$$
Since each $Z\in\mathscr{A}_{j+1}$ is of the same degree, there cannot be a relation among these characteristics. Thus, each $e^t(\pi_{\tZ})$ for $\tZ$ an upper stratum element in $\mathscr{A}_{j+1}$ must be 0. It follows that, for $\eta$ the unit map from $\S$ to $A(*)$, $\eta_* e^d_\partial(p)=e^t_\partial(p)=0$, and $\eta_* e^d(\pi_\tZ)=e^t(\pi_\tZ)=0$ for each $Z\in\sS(\Sigma_\sd^\#)$.   

Since the fibration $\Gamma_B\h_B^\%(W)\to\Gamma_B Q_B^\%(W)\to\Gamma_B A_B^\%(W)$ is split by the trace map, we have a short exact sequence

$$
0\to\pi_0\Gamma_B\h^\%_B(W)\otimes\Q\to\pi_0\Gamma_B Q_B(W)\otimes\Q\to\pi_0\Gamma_B A^\%_B(W)\otimes\Q\to0
$$

\medskip\noindent It then follows from Definition \ref{defn:euler_H} that 
$$
e^{t/d}_\partial(p) = \displaystyle\sum_{Z_j\in \mathscr{A}_j} (-1)^j e^{t/d}(\pi_{\tZ_i}) + \displaystyle\sum_{Z_{j+1}\in \mathscr{A}_{j+1}} (-1)^{j+1} e^{t/d}(\pi_{\tZ_{j+1}})
$$
The result follows since $\mathscr{A}_{j}\cup\mathscr{A}_{j+1}=\sS(\Sigma_\sd^\#)$
\end{proof}

\section{Calculations of the smooth structure class}
\label{sec:ThetaCalc}

In this section we prove the duality and vanishing theorems. We begin in subsection~\ref{subsec:setup_conj_pf} by reviewing the setup for the proofs of these theorems. In Subsection~\ref{subsec:duality_thm} we use Theorem~\ref{thm:PHTheta} to prove a duality theorem for the smooth structure class, Theorem~\ref{thm:duality}. In Subsection~\ref{subsec:conj_proof} we prove the vanishing theorem, Theorem~\ref{thm:rigidity}. We also give slight generalizations to manifolds with boundary in Theorem~\ref{thm:rigidity_relative} and Corollary~\ref{cor:rigidity_stable}.

\subsection{Setup for the proofs of the duality and vanishing theorems}
\label{subsec:setup_conj_pf}

Let $W$ be a smooth h-cobordism bundle over $B$ that is \textit{topologically trivial} as in Section ~\ref{sec:2}. 
This means that $p:W\to B$ is a smooth fiber bundle with two boundary components $\partial_0W:=M_0$ and $\partial_1W:=M_1$ so that $p_0:M_0\to B$ and $p_1:M_1\to B$ are smooth manifold bundles and there is a homeomorphism $h:W\to M\times I$ over $B$.
If the fibers of $M_0$ and $M_1$ themselves have boundary, then we additionally require that $\partial^\vee W$ is diffeomorphic to $M_0\cup \partial M_0\times I\cup M_0$.

Let $\chi^\%(W,\partial_0W)$ denote the \textit{relative excisive A-theory Euler characteristic} of $(W,\partial_0W)$, and let $\tr(W,\partial_0W)$ denote the \textit{relative Becker–Gottlieb transfer} of $(W,\partial_0W)$. 
Then the map $h$ provides a path from $\chi^\%(W,\partial_0W)$ to  $\chi^\%(M_0\times I, M_0)$. 
Concatenating this path with the homotopy from $\eta\circ\tr(W,\partial_0W)$ to $\chi^\%(W,\partial_0 W)$ supplied in~\cite{DWW03}, we have a nullhomotopy of the composition $\eta\circ\tr(W,\partial_0W)$. 
In Definition~\ref{defn:theta}, the \textit{relative smooth structure characteristic} $\theta(W,\partial_0W)\in \Gamma_B\H^\%_B(W)$ is defined to be the lift of $\tr(W,\partial_0W)$ determined by this nullhomotopy. 
This is summarized in the following diagram. 

$$
\xymatrix{
&&\Gamma_B\H_B^\%(W)\ar[dd]\\
\\
S^0\ar[rruu]^-{\theta(W,\partial_0W)}\ar[rr]^-{\tr(W,\partial_0W)}\ar[ddrr]_{\chi^\%(W,\partial_0W)}&&\Gamma_BQ_B (W)\ar[dd]^{\eta}\\
\\
&&\Gamma_BA_B^\%(W)
}
$$

\noindent We can similarly define $\theta(W,\partial_1W)$ in $\Gamma_B\H_B^\%(W)$.

\begin{prop}\label{prop:PDtheta}
Given a topologically trivial family of smooth h-cobordisms $p:W\to B$, the smooth structure characteristic $\theta(W,\partial_0W)$ is fiberwise Poincar\'e dual to the Euler section $e^{t/d}_\partial(p)$ of Definition~\ref{defn:euler_H}. 
\end{prop}
\begin{proof}
This follows by combining the definition of $\theta(W,\partial_0W)$ (\ref{defn:theta}) as the homotopy fiber of $\tr(W,\partial_0W)$ over $\chi^\%(W,\partial_0W)$, the definition of $e^{t/d}_\partial(p)$ as the homotopy fiber of $e^d_\partial(p)$ over $e^t_\partial(p)$,  and Proposition \ref{prop:DWW_index_thms}.
\end{proof}

\noindent To prove the vanishing theorem, we must compute $\Theta(W,\partial_0W)-\Theta(W,\partial_1W)$ for a topologically trivial h-cobordism bundle $W$. 

\subsection{A Duality Theorem for the smooth structure class}
\label{subsec:duality_thm}

In this section we prove a duality theorem for the smooth structure class using Theorem~\ref{thm:PHTheta}.
Recall how, in Section~\ref{sec:PHthms}, we made use of a fiberwise generalized Morse function $f:W\to [0,1]$ for which $f(\partial_0N)=0$ and $f(\partial_1N)=1$. 
In this section we will also make use of the fiberwise generalized Morse function $\overline{f} := 1-f$. 
In particular we will prove a duality theorem, Theorem~\ref{thm:duality}, by applying Theorem~\ref{thm:PHTheta} to $f$ and $\overline{f}$, and then comparing the results.

\begin{theorem}
\label{thm:duality}
For $p:W\to B$ a topologically trivial bundle of smooth h-cobordisms with fiber dimension $n$,
$$
\Theta(W,\partial_0W) = (-1)^{n-1}\Theta(W,\partial_1W)
$$
\end{theorem}

\begin{proof}[Proof of Theorem~\ref{thm:duality}]

From Subsection~\ref{subsec:strat_def}, we have a stratified deformation of the stratified subset $(\Sigma_f,\psi_f)$ to the stratified subset concentrated in two degrees $(\Sigma_\sd,\psi_\sd)$. Applying this stratified deformation instead to $(\Sigma_{\overline{f}},\psi_{\overline{f}})$, we obtain the stratified deformation $(\overline{\Sigma_\sd},\overline{\psi_\sd})$. We take care to point out that while $\Sigma_\sd$ and $\overline{\Sigma_\sd}$ are diffeomorphic, a particular stratum of $\sS(\Sigma_\sd)$ may correspond to a stratum of a different degree in $\overline{\Sigma_\sd}$, and the bundles $\psi_\sd$ and $\overline{\psi_\sd}$ are only equivalent after applying the swap map $BO\times BO\to BO\times BO$, since the positive and negative eigenspaces have been exchanged. 

\medskip

\noindent We begin by applying Theorem~\ref{thm:PHTheta} to $f$ and $\overline{f}$ to obtain the following formulas. 

\begin{equation}\label{eqn:posf}
e^{t/d}_{\partial_0}(p) = \displaystyle\sum_{Z_i\in\sS(\Sigma^\#_\sd)} (-1)^i e^{t/d}(\pi_{\tZ_i})
\end{equation}

\begin{equation}\label{eqn:negf}
    e^{t/d}_{\partial_1}(p) = \displaystyle\sum_{Z_j\in\sS(\overline{\Sigma^\#_\sd})} (-1)^j e^{t/d}(\pi_{\tZ_i})
\end{equation}


The stratified subsets $(\Sigma_{SD},\psi_{SD})$ and $(\overline{\Sigma_{SD}},\overline{{\psi}_{SD}})$ are each disjoint unions of immersed lenses concentrated in two degrees, and a component on which the tangential data is trivial. 
We do not consider this extra component because it has a trivial rational contribution, according to Lemma~\ref{lem:null_deformable}.
Because the critical loci of $f$ and $\overline{f}$ are identical, the submanifolds $\Sigma_{SD}$ and $\overline{\Sigma_{SD}}$ constructed from the same stratified deformation are identical. 
In particular, there is a one-to-one correspondence between immersed lenses comprising $(\Sigma_{SD},\psi_{SD})$ and $(\overline{\Sigma_{SD}},\overline{{\psi}_{SD}})$. 
It suffices to consider one such pair, and observe how the bundle data $\psi_{SD}$ and $\overline{{\psi}_{SD}}$ has changed. 
Below we depict $L_{i-1}(Z,\psi_{i-1},\psi_i)$, an immersed lens belonging to $(\Sigma_{SD},\psi_{SD})$. 

\begin{center}
\begin{figure}[H]
\begin{overpic}[abs,width=8cm]{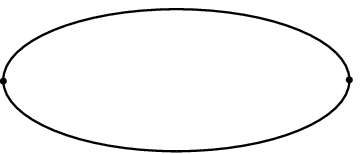}
\put(-100,45){$L_{i-1}(Z,\psi_{i-1},\psi_i)$}
\put(-30,80){degree $i$}
\put(-30,10){degree $i-1$}
\put(220,10){$\psi_{i-1}=*$}
\put(220,80){$\psi_{i}$}
\put(110,10){$Z$} 
\put(110,80){$Z$}
\end{overpic}
\end{figure}
\end{center}

\noindent The immersed lens $L_{i-1}(Z,\psi_{i-1},\psi_i)$ above corresponds to the immersed lens $L_{n-i}(Z,\psi_{n-i},\psi_{n-(i-1)})$ belonging to $(\overline{\Sigma_{SD}},\overline{{\psi}_{SD}})$, depicted below.

\begin{center}
\begin{figure}[H]
\begin{overpic}[abs,width=8cm]{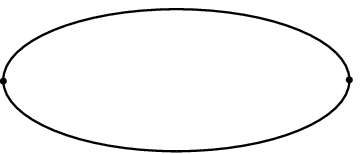}
\put(-125,45){$L_{n-i}(Z,\psi_{n-i},\psi_{n-(i-1)})$}
\put(-65,80){degree $n-(i-1)$}
\put(-65,10){degree $n-i$}
\put(220,10){$\psi_{n-i}=*$}
\put(220,80){$\psi_{n-(i-1)}$}
\put(110,10){$Z$}
\put(110,80){$Z$}
\end{overpic}
\end{figure}
\end{center}

Recall that after the stratified deformation of Lemma~\ref{lem:SD1}, the bundle data on the lower stratum is trivial, so $\psi_{i-1}$ and $\psi_{n-i}$ are trivial as in the diagrams above. 
On the upper strata, $\psi_{SD}$ is the map $Z\to BO\times BO$ classifying the stable negative and positive eigenspace bundles, $\gamma_f$ and $\gamma_{-f}$, respectively. These stable bundles $\gamma_f$ and $\gamma_{-f}$ have the property that $\gamma_f\oplus\gamma_{-f}\cong T^\vee M|_Z\oplus \epsilon^n$. 
The section $e^{t/d}(\pi_\tZ)$ is defined in terms of the restriction of the vertical tangent bundle $T^\vee M|_Z$.
Thus the summand $(-1)^i e^{t/d}(\pi_{\tZ_i})$ for $Z\in\sS(\Sigma_\sd^\#)$ is the same, up to a sign, as the summand $(-1)^{n-(i-1)}e^{t/d}(\pi_{\tZ_{n-(i-1)}})$ with $Z_{n-(i-1)}$ in $\overline{\Sigma_\sd^\#}$. It then follows that 

$$
e^{t/d}_{\partial_0}(p) = (-1)^{n-1}e^{t/d}_{\partial_1}(p)
$$

\noindent Taking Poincar\'e duals on both sides gives the result in the theorem statement.
\end{proof}

\subsection{Proof of the vanishing theorem}
\label{subsec:conj_proof}
We can now prove the Vanishing Theorem.

\begin{theorem}\label{thm:rigidity}
If the fibers of $p_0:M\to B$ are even dimensional and closed, then for a topologically trivial family of smooth h-cobordisms $W$ from $M$ to $M'$, $\Theta(W,M)-\Theta(W,M')$ is trivial. 
\end{theorem}

\begin{proof}
When the fibers of $M$ are even dimensional, $n-1$ is even. 
Thus, by Theorem~\ref{thm:duality}, $\Theta(W,\partial_0W) = \Theta(W,\partial_1W)$. 
\end{proof}

We might also consider the relative case, where the boundaries $\partial_0W$ and $\partial_1W$ have corners.
Recall that in this case we consider h-cobordisms $W$ from $M$ to $M'$ so that $\partial^\vee W$ is diffeomorphic to $M\cup \partial M\times I \cup M'$. 

\begin{theorem}\label{thm:rigidity_relative}
If $M$ has boundary and the fiber dimension of $M$ is even, then for any topologically trivial family of smooth h-cobordisms $W$ from $M$ to $M'$ with $\partial^\vee W$ diffeomorphic to $M\cup \partial M\times I \cup M'$, the smooth structure class $\Theta(W,M)-\Theta(W,M')$ is trivial. 
\end{theorem}

\begin{proof}
The proof is the same as the proof of Theorem~\ref{thm:rigidity}.
\end{proof}

\noindent Applying the two theorems above to a bundle $M$ with closed fibers, we obtain the following corollary. 

\begin{corollary}\label{cor:rigidity_stable}
If $n+k$ is even and $k\ge 0$, then for any topologically trivial family of smooth h-cobordisms $W$ from $M\times I^k$ to $M'\times I^k$ satisfying the condition $\partial^\vee W =M\times I^k\cup \partial (M\times I^k)\times I \cup M'\times I^k$, the smooth structure class $\Theta(W,M'\times I^{k})-\Theta(W,M'\times I^{k})$ is trivial. 
\end{corollary}

\begin{rmk}
Note that the dependence on $k$ in the corollary above implies that the statement applies just as well when $n$ is odd and $k$ is odd. 
We explain why this does not contradict the constructions in~\cite{GI14}, which produce topologically trivial h-cobordisms whose boundaries are bundles with closed odd dimensional fibers. 
If we take such a bundle and stabilize once by multiplying with an interval, we produce an h-cobordism, but the boundary is not a product h-cobordism, and thus Theorem~\ref{thm:rigidity_relative} does not apply. 
It may be helpful to note that the difference $\Theta(W,\partial_0W)-\Theta(W,\partial_1W)$ is not a stable invariant with respect to the upper and lower stabilization maps on the h-cobordism space, which maintain a product structure on the boundary. 
In particular, $\Theta(W,\partial_0W)$ is an invariant of the lower stabilization map, and $\Theta(W,\partial_1W)$ is an invariant of the upper stabilization map. 
Still, the difference $\Theta(W,\partial_0W)-\Theta(W,\partial_1W)$ is not an invariant of either stabilization map, but is only an invariant of the stabilization which multiplies the entire h-cobordism by an interval.
Theorem \ref{thm:rigidity_relative} does not apply to h-cobordisms obtained from that type of stabilization.
\end{rmk}

\newpage
\bibliography{ref}
\bibliographystyle{alpha}
\nocite{*}

\end{document}